\documentclass[a4paper,11pt]{article}

\usepackage{graphicx}
\usepackage{bm}
\usepackage{arydshln}
\usepackage{dashrule}
\usepackage{url}
\usepackage[percent]{overpic}
\usepackage{authblk}
\usepackage{amscd}
\usepackage{amsfonts}
\usepackage[numbers]{natbib}
\usepackage{amsthm}
\usepackage{amssymb}
\usepackage{amsmath}
\usepackage{algpseudocode}
\usepackage{algorithm}
\usepackage{cancel}
\usepackage{chngcntr}
\usepackage{caption}
\usepackage{enumerate}
\usepackage{epsfig}
\usepackage{epic}
\usepackage{eepic}
\usepackage{enumitem}
\usepackage{etoolbox}
\usepackage{fancyhdr}
\usepackage{fullpage}
\usepackage[a4paper, total={6in, 8in}]{geometry}
\usepackage{graphicx}
\usepackage{here}
\usepackage[utf8]{inputenc}
\usepackage{lscape}
\usepackage{listings}
\usepackage{longtable}
\usepackage{natbib}
\usepackage{nccmath}
\usepackage{parskip}
\usepackage{rotating}
\usepackage{subfigure}
\usepackage{tabularx}
\usepackage{textcomp}
\usepackage{tikz}
\usepackage{titlesec}
\usepackage{threeparttable}
\usepackage{varwidth}
\usepackage{framed,multirow}
\usepackage{todonotes}
\newtheorem{pro}{Proposition}

\newtheorem{remark}{Remark}
\usepackage{comment}
\usepackage{todonotes}
\usepackage{mathtools}
\usepackage{amsmath,amssymb}
\usepackage{graphicx}
\usepackage[percent]{overpic}
\usepackage{arydshln} % For dashed lines
\usepackage[normalem]{ulem}
\definecolor{darkgreen}{rgb}{0.2, 0.5, 0.0}

\usepackage{xcolor}

\title{Time multiscale modeling of sorption kinetics I: uniformly accurate schemes for highly oscillatory advection-diffusion equation}

\author[1]{Clarissa Astuto}
\author[2]{Mohammed Lemou}
\author[3]{Giovanni Russo}

\affil[1]{King Abdullah University of Science and Technology (KAUST), 4700, Thuwal, Saudi Arabia}
\affil[2]{IRMAR, CNRS et Universit\'{e} de Rennes 1, France}
\affil[3]{Department of Mathematics, University of Catania, Italy}

\begin{document}
\maketitle

\begin{abstract}
In this paper we propose a numerical method to solve a 2D advection-diffusion equation, in the highly oscillatory regime. We use an efficient and robust integrator which leads to an accurate approximation of the solution without any time step-size restriction. %[check that this is indeed the case and that there is no CFL type restriction]}. 
Uniform first and second order numerical approximations in time are obtained with errors, and at a cost, that are independent of the oscillation frequency. 
{This work is part of a long time project, and the final goal is the resolution of a Stokes-advection-diffusion system, in which the expression for the velocity in the advection term, is the solution of the Stokes equations.} 
This paper focuses on the time multiscale challenge, coming from the velocity that is an $\varepsilon-$periodic function, whose expression is explicitly known. We also introduce a two--scale formulation, as a first step to the numerical resolution of the complete oscillatory Stokes-advection-diffusion system, that is currently under investigation. This two--scale formulation is also useful to understand the asymptotic behaviour of the solution. 
\end{abstract}

	\section{Introduction}
\label{sec:introduction}
The diffusion of particles, in presence of moving traps, is an interesting topic in different areas, from biology \cite{variA,variC,variB,Lotka-Volterra}  to chemistry \cite{Benichou,pierresinserm00416601}, with a relevant application to the study of the relation between living cell membranes and substances freely diffusing around them \cite{Raudino20168574,fernandez2012mixed,SIAM1_fernandez2016existence}.

%%% ---------------------
With the aim of studying the capture rate (chemoreception), a biomimetic model has been developed \cite{corti2014out,corti2015trapping,Raudino20168574}, in which an oscillating air bubble mimics a fluctuating cell, and a flow of charged surfactants simulates the diffusing substances (see Fig.~\ref{fig_delta} (a)). Surfactants are composed by anions and cations, that have two different configurations. The cations are hydrophilic, while the negative ions  have both hydrophilic and hydrophobic parts, and, for this reason, they are absorbed at the air-water interface %. The hydrophobic part prefers to stay in the air, while the hydrophilic part stays in water 
(see Fig.~\ref{fig_delta} (a)). 
%%%% ------------

The problem of surfactant diffusion or molecules, that are adsorbed at the surface of a moving cell, has been investigated by several authors \cite{Raudino20168574,astuto2023multiscale,CiCP-31-707,SIAM2_ganesan2012arbitrary,SIAM3_morgan2015mathematical,SIAM4_xu2013analytical,WIEGEL1983283,BERG1977193} %and, in each work, the transport phenomena in ionic solutions requires the calculation of a drift-diffusion equation of the concentration of ions (see \cite{}). 
and the starting model for the evolution of a single species carriers is the one described in \cite{astuto2023multiscale}, where the authors introduced the local concentration of 
ions\footnote{In reality there are two types of ions, anions and cations, each with its own concentration, mutually attracting each other by Coulomb interaction. In the so called quasi-neutral limit the two concentrations completely overlap, and the system can be described by just one scalar concentration, see \cite{astuto2023multiscale}.} $c= c(\vec{x},t)$, whose time evolution in a static fluid is governed by the conservation law
\begin{equation}
	\displaystyle \frac{\partial c}{\partial t}=-\nabla\cdot J.
	\label{equation_flux}
\end{equation}
In presence of a static bubble (see Fig.~\ref{fig_delta} (a)), and for small concentrations $c$, the flux has the following expression
\begin{eqnarray}
	\displaystyle J=\ -D\left(\nabla c +\ \frac{1}{k_BT}c\nabla V\right)
	\label{equation_definition_flux}	
\end{eqnarray}
where $D$ is the diffusion coefficient, $k_B$ is the Boltzmann's constant, $T$ is the absolute temperature and $V$ is a suitable potential function that describes the \textit{attractive-repulsive} behavior of the bubble with the particles (see Fig.~\ref{fig_delta} (b)).

\begin{figure}
	\centering
	\begin{minipage}[b]
		{.49\textwidth}
		\centering
		\begin{overpic}[abs,width=\textwidth,unit=1mm,scale=.25]{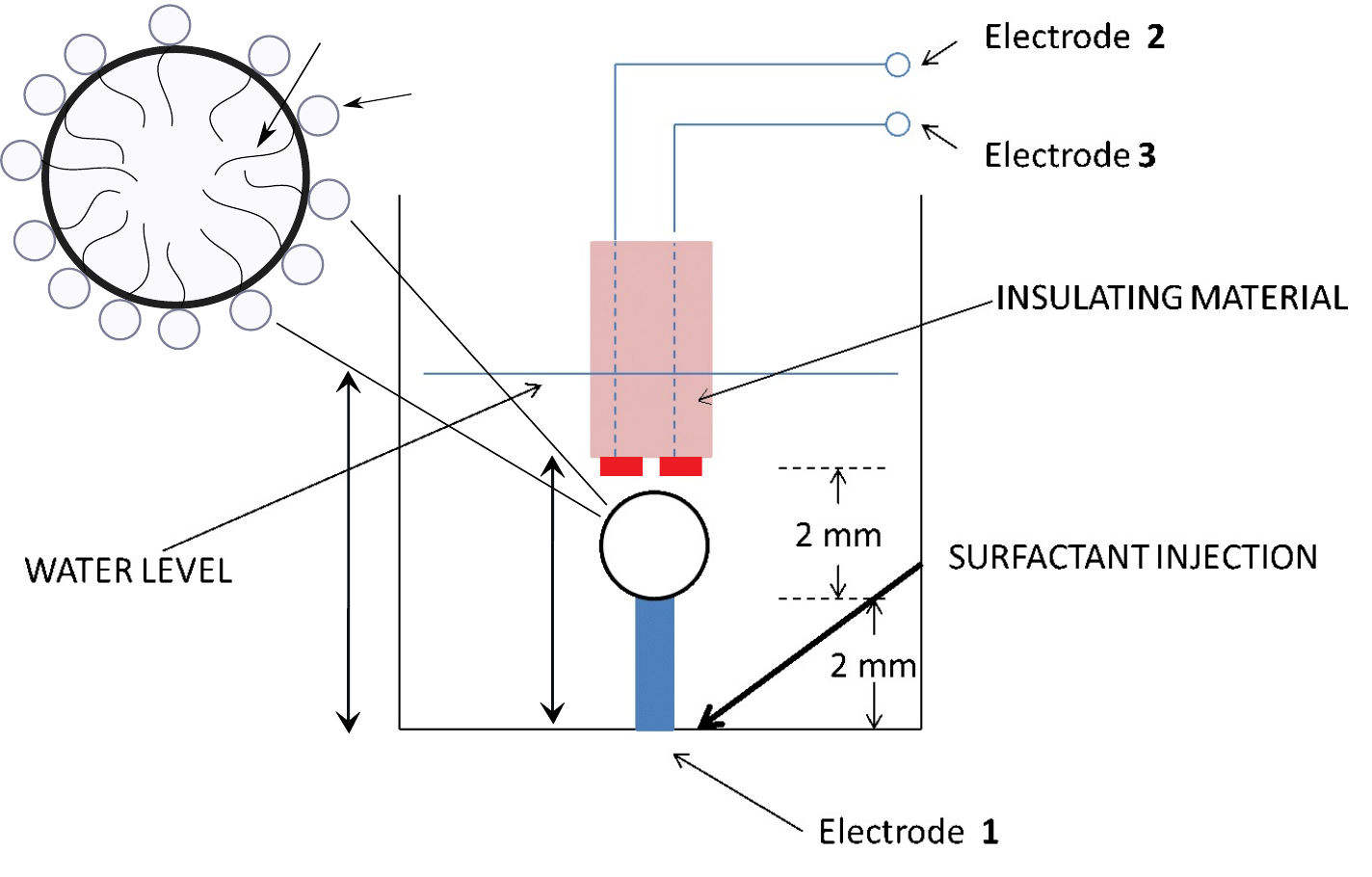}
			\put(14,8){{0}}
			\put(14,24){{H}}
			\put(26,12){{P}}
			\put(-4,45){(a)}
			\put(14,46.5){\footnotesize hydrophobic tail}
			\put(19,43){\footnotesize hydrophilic}
			\put(23,40){\footnotesize head}
		\end{overpic}
	\end{minipage}\hfill
	\begin{minipage}[b]{.49\textwidth}
		\begin{overpic}[abs,width=0.85\textwidth,unit=1mm,scale=.25]{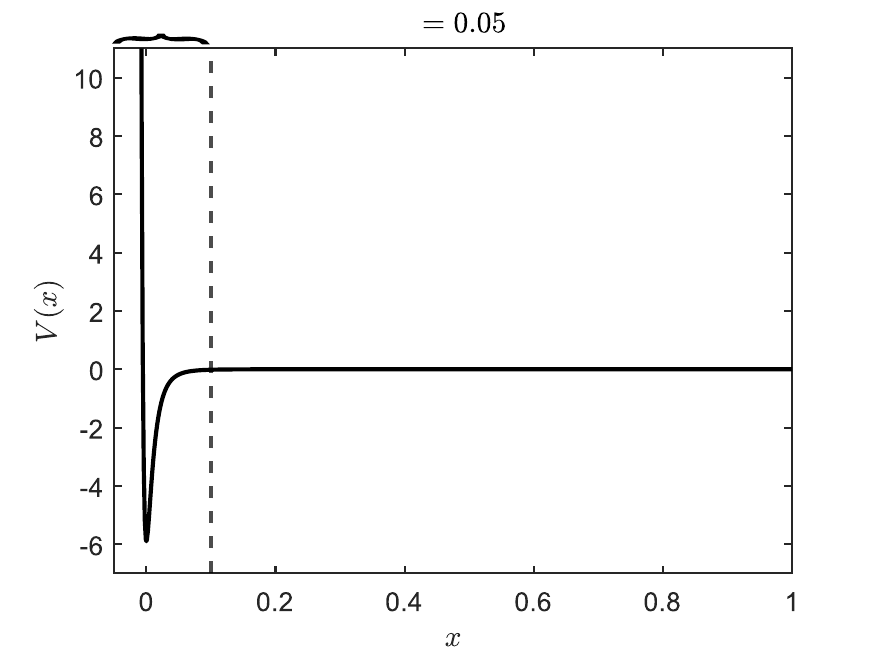}
			\put(28,45){{$\delta$}}
			\put(9.5,46.5){{$\Omega^\delta_{\rm b}$}}
			\put(46,46.5){{$\Omega_{\rm f}$}}
			\put(-1.,45){(b)}
		\end{overpic}
		% \centering
		% \includegraphics[width=0.85\textwidth]{Figures/V_x_delta_05.pdf}
	\end{minipage}
	\caption{\textit{(a): Scheme of the experimental setup.  On the top left there is a zoom in of the anions behaviour at the air-surface of the bubble, with the hydrophobic tails inside the air bubble, and the hydrophilic heads on the surface. The water level is H and the detector is located in a point P. (b) Scheme of the potential $V(x)$, defined in Eq.~\eqref{eq_V_expr}, where $\delta$ is the thickness of the attractive-repulsive  part.}}
	\label{fig_delta}
\end{figure}
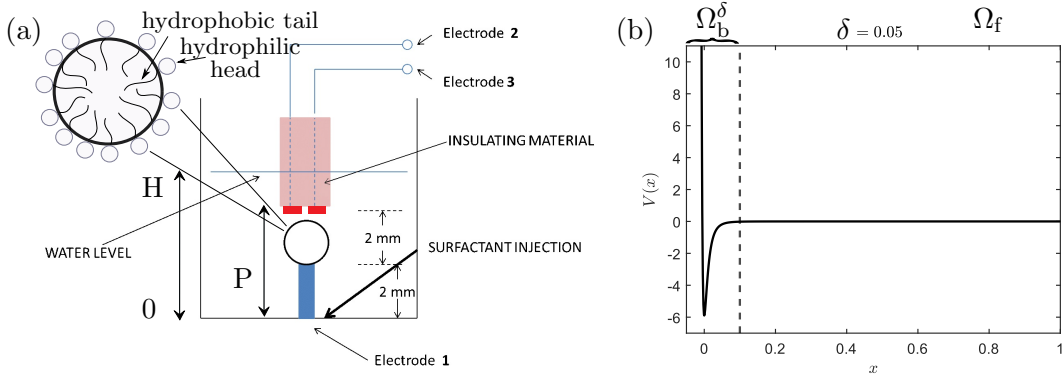

%% ----------------------------------------------
For simplicity, here we describe only the one space dimension model (see \cite{astuto2023multiscale} for more details and dimensions). We assume that the fluid domain, which is not affected by the bubble, is $\Omega^\delta_{\rm f} = [L\delta,1]$, and that the attractive-repulsive mechanism of the bubble is simulated inside a thin region $\Omega^\delta_{\rm b}= [-\delta,L\delta]$, with $L$ a constant of order 1. 
It means that the potential $V(x)=0$ for $x \in \Omega^\delta_{\rm f}$.

In particular, the particles, that are close to the bubble, are attracted by its surface, but, at the same time, when an attracted particle is at very short distance from the surface of the bubble, it is repulsed, so that it cannot enter inside the bubble. This repulsion simulates the impermeability of the air bubble, and, in \cite{astuto2023multiscale} we choose the Lennard-Jones potential (LJ) as a prototypical attractive-repulsive potential, as follows
\begin{eqnarray}
	\label{eq_V_expr}
	V(x) = E\left( \left(\frac{x+\delta}{\delta}\right)^{-12} - 2\left(\frac{x+\delta}{\delta}\right)^{-6} \right),
\end{eqnarray}
where $\delta$ denotes the range of the potential and $E$ represents the depth of the well, (see Fig.~\ref{fig_delta} (b)).

In \cite{astuto2023multiscale} we proposed a \textit{multiscale model}, based on asymptotic expansion in $\delta$, to describe this adsorption-desorption behaviour. The space multiscale nature derives from the potential that is not negligible only in $\Omega^\delta_{\rm b}$, which is very small compared to the entire domain. 

%leading to some extra computational effort to ensure that the correct behaviour is captured. 
%In order to overcome this difficulty, we propose a multiscale model for $\varepsilon \ll 1$ in which the domain of the problem is reduced to $\Omega^\varepsilon_f$ and the effect of $V(x)$ on $\Omega^\varepsilon_b$ is approximated by a suitable boundary condition at the origin and obtained as follows. 

Summarizing, the 1D {multiscale model} for a single carrier, in the low concentration approximation, can be obtained for $\delta \ll 1$ (in general, for $\delta/L_x \ll 1$ where $L_x$ is the length of the 1D domain) as follows:
\begin{align}\label{reduced1d}
	\frac{\partial  c }{\partial t} &= D\frac{\partial^2  c }{\partial x^2} \quad {\rm {in} }\, x \in [0,1]\\
	\frac{\partial  c }{\partial x} &= 0  \quad {\rm {at} }\, x = 1, \qquad \mathcal{M}\frac{\partial  c }{\partial t} = D\frac{\partial  c }{\partial x}  \quad {\rm {at} }\, x = 0
\end{align}
and
\begin{equation}
	\label{expr_M}
	\mathcal{M} = \delta\int_{0}^{L+1}\exp\left(-U(\zeta)\right)d\zeta.
\end{equation}
where $U(\zeta)= \phi \left( \zeta^{-12} - 2\zeta^{-6} \right)$ is a non dimensional form of the potential $V(x)$, with $\zeta = 1 + x/\delta \in [0,L+1]$ is a rescaled variable, $L$ is the distance at which the potential $U$ is negligible and $\displaystyle \phi = {E}/{k_BT}$. In our case, we pose $L=2$.

%% ------------------------------------------------

In this scenario, it is important the role played by oscillating traps with particles, as well as, the computation of their adsorption and desorption rates at the surface of the trap. In \cite{ASTUTO2023111880}, we add an advection term to the diffusion equation, to take into account the movement of the fluid, due to the oscillations of the bubble. Then we couple the advection-diffusion equation for the particle concentration, with the Stokes equations for the fluid velocity, and we
consider different kinds of oscillations, e.g., radial oscillations or shape deformations, where we show how the diffusion rate depends on the type of oscillations and on its frequency. %Particularly relevant is the adsorption rate when the bubble is exposed to intense forced oscillations, {\color{red} near resonance}.
%\todo[inline]{aggiungere riferimento?}
%Bubble oscillations of the order of a few nanometers are selectively excited, and surfactant transport is accurately measured \cite{Raudino20168574}. The oscillation frequency is of the order of hundredths \textit{Hz}, while the diffusing time is of the order of hours: these two different scales in time introduce in the model a multiscale challenge. 
%The surfactant concentration past the oscillating bubbles was detected by conductivity measurements \cite{Raudino20168574}. The results highlight the role of surface oscillations on the diffusant capture rate. Particularly unexpected is the overshoots during the adsorption process. The phenomenon is particularly relevant when the bubbles are exposed to intense forced oscillations near resonance.

In this work, we focus on the time multiple scale nature of the problem. The technique we use is based on the assumption that the motion of the bubble, and the fluid, are periodic in time, with a period $\varepsilon$ that is much shorter than typical diffusion times. Indeed, in laboratory experiments, bubble oscillations frequency is of the order of hundred {Hz}, while the diffusing time is of the order of hours (see \cite{Raudino20168574,CiCP-31-707} for more details).

Standard numerical methods, for the Eq.~\eqref{eq_ode}, produce errors of the order $\Delta t^p/\varepsilon^q$, for some positive $p$ and $q$. In that way, the user is forced to obey to the restriction on the time step to obtain the desired accuracy, i.e., $\Delta t \lesssim \varepsilon^{q/p}$ . This restriction becomes prohibitive for small values of $\varepsilon$. We will follow the strategy adopted in \cite{chartier2015uniformly,chartier2020new,chartier2022derivative,crouseilles2013asymptotic,crouseilles2017uniformly} although in the different contexts of Vlasov--Poisson equations, Klein--Gordon and nonlinear Schr\''odinger equations, to obtain a robust scheme that is able to deal with a large range of $\varepsilon \in (0,1]$ (being small or not), since our goal is to obtain a numerical scheme that is uniformly accurate in $\varepsilon$.

\section{Model}
\label{sec:model}
In this paper we investigate the time multiscale coming from the advection term in the system. As a preliminary, work we assume that the expression for the velocity is known, and we do not have to compute the solution of the Stokes equations.% \giovanni{toglierei questo commento, con ambigua interpretazione (as we did in \cite{ASTUTO2023111880})}. 

In the presence of a moving fluid, the conservation law for the local concentration of ions $c = c(\vec{x},t)$ is the same as \eqref{equation_flux}
\begin{equation}
	\frac{
		\partial c}{\partial t}=-\nabla\cdot \vec J, \quad {\rm in }\, \mathcal{S},
\end{equation}
where $\mathcal{S} \subset \mathbb{R}^2$. However, this time the flux term $\vec J$ contains a diffusion and an advection term, %\giovanni{Abbiamo usato $T$ per temperatura e tempo finale. Per piacere , sistema...}
\begin{equation} \label{eq:flux}
	\vec J=\ -D\nabla c - c\, \vec{u}, \quad t \in [0, t_{\rm fin}]
\end{equation}
where $t_{\rm fin}>0$, $D$ is the diffusion coefficient, $\vec{u} = \vec{u}(\vec x,t/\varepsilon) \in \mathbb{R}^2$ is explicitly known in space and time, and is assumed to be a periodic vector function of time with period equal to  $\varepsilon \in ]0, \varepsilon_0]$, for some $\varepsilon_0 > 0$. We add a subscript $\varepsilon$ on the concentration $c_\varepsilon = c$, to emphasize its dependence on the oscillation period, and at the end, the system reads
\begin{equation}
	\label{eq_ode}
	\frac{\partial c_\varepsilon}{\partial t} = D\Delta c_\varepsilon + \nabla \cdot (c_\varepsilon \vec{u}(t/\varepsilon)),  \quad {\rm in }\, \mathcal{S}.
\end{equation}
{From now on, we omit to indicate the explicit dependence of $\vec u$ on space, while we keep its dependence on $t/\varepsilon$.}
Let $\mathcal{S}$ be a square, and  $\Omega = \mathcal{S}\setminus \mathcal{B}$ the computational domain where $\mathcal{B}$ is a circle centered in $(0,0)$, with radius $R_{\mathcal{B}}$ (see Fig.~\ref{2Ddomain_points} (a)). The boundary of the domain is defined as $\Gamma = \partial \Omega = \Gamma_\mathcal{S} \cup \Gamma_\mathcal{B}$; see Fig.~\ref{2Ddomain_points} (a). 

Eq.~\eqref{eq_ode} is completed with homogeneous Neumann boundary conditions in $\Gamma_\mathcal{S}$ and absorption-desorption boundary conditions in $\Gamma_\mathcal{B}$ (see \cite{astuto2023multiscale} for more details), i.e., in other words
\begin{eqnarray}
	\displaystyle \nabla  c_\varepsilon \cdot n &=& 0 \quad  {\rm  on }\, \Gamma_\mathcal{S}\\ \label{eq_bc_M_2D}
	\displaystyle \mathcal M\frac{\partial  c_\varepsilon}{\partial t} &=& \mathcal M D \frac{\partial ^2 c_\varepsilon}{\partial \tau ^2}-D\frac{\partial  c_\varepsilon}{\partial n }\quad \text{ on $\Gamma_\mathcal{B}$},
\end{eqnarray}
%\giovanni{L'eq. (2.5) è sbagliata. Sarebbe corretta solo in 2D. Mi dispiace, ma i test numerici sono da rifare. L'operatore di Laplace-Beltrami sulla superficie della sfera in simmetria cilindrica diventa:
	where $n$ is the outgoing normal vector to $\Gamma$, and $\tau$ is the tangent vector to $\Gamma_\mathcal{B}$. Eq.~\eqref{eq_bc_M_2D} is the analogue expression of Eq.~\eqref{expr_M}, but in higher dimension. %A multiscale single carrier model in space has been derived in \cite{astuto2023multiscale}, which describes the interaction of the bubble on the particles (anions) by a suitable boundary condition for the advection-diffusion equation in $\Gamma_{\mathcal{B}}$, derived by mass conservation and asymptotic analysis in the region near the bubble.
	
	To close the system~(\ref{eq_ode}-\ref{eq_bc_M_2D}), we add an initial condition 
	\begin{eqnarray}
		\label{eq_IC}
		c _\varepsilon(0) &=&  c _\varepsilon^0
	\end{eqnarray}
	that does not depend on $\varepsilon$.
	%Standard numerical methods for Eq.~\eqref{eq_ode} produce errors of the order $\Delta t^p/\varepsilon^q$, for some positive $p$ and $q$. In this way, we are forced to obey to the restriction on the time step, such that $\Delta t < \varepsilon^{q/p}$ to obtain the desired accuracy. This restriction becomes prohibitive for small values of $\varepsilon$. We will follow the strategy showed in \cite{chartier2015uniformly,crouseilles2013asymptotic} to obtain a robust scheme that is able to deal with negligible values of $\varepsilon$, since our goal is to obtain a scheme that is uniformly accurate in $\varepsilon$.
	
	Considering a numerical scheme of order $q > 1$, it means by definition that for all $\varepsilon > 0$, there exists a constant $K(\varepsilon)$, and a time-step $\overline{\Delta t}(\varepsilon)$ such that, for all $\Delta t < \overline{\Delta t}(\varepsilon)$, the error $E_\varepsilon(\Delta t)$ is bounded by 
	\begin{equation}
		\label{eq_order_q}
		E_\varepsilon(\Delta t) < K(\varepsilon)\Delta t^q.    
	\end{equation}
	
	The goal of this paper is to construct numerical schemes that are stable, for all $\varepsilon \in (0, 1]$, and whose order does not depend on $\varepsilon$ and does not degrade when $\varepsilon \to 0$. In other word, there exist a constant $K$, independent of $\varepsilon$, and a time-step $\overline{\Delta t}$ such that, for all $\Delta t < \overline{\Delta t}$, the error $E(\Delta t)$ is bounded by 
	\begin{equation*}
		E(\Delta t) < K\Delta t^q,   
	\end{equation*}
	where $q$ is the same as in Eq.~\eqref{eq_order_q}.
	
	To show that the numerical schemes, proposed in this paper, are uniformly accurate in $\varepsilon$, we first observe that there are not oscillations of lengthscale proportional to $\varepsilon$ in space.

	\begin{remark}
		\label{remark_space_osc}
		Since $\vec{u}(t/\varepsilon)$ depends on $\vec{x}$ but not on $\vec{x}/\varepsilon$, the solution $c_\varepsilon$ is not oscillatory in space. This is attested by numerical tests, see Figs.~\ref{fig_detector_cos}. %-\ref{fig_detector_cos_sin}. 
		However, if $u$ oscillates in $\vec{x}$ (which means that it has some smooth dependence on $\vec{x}/\varepsilon$), then the solution $c_\varepsilon$ is oscillating in space as well. This is also observed in Fig.~\ref{fig_space_osc}.
	\end{remark}
	%The strategy that we use in this paper (explained also in \cite{chartier2015uniformly,crouseilles2013asymptotic}) consists in separating the two time scales, namely the slow one $t$ and the fast one $t/\varepsilon$, and, in Section~\ref{section_two_scales}, we consider a change of variable of the solution $c_\varepsilon(t)$, into a two-variable function $C(t, \Theta)$, under the constraint that $C(t,\Theta) = c_\varepsilon(t)$ when $\Theta = t/\varepsilon$.
	\section{Description of the spatial domain}
	\label{section_discr_space}
	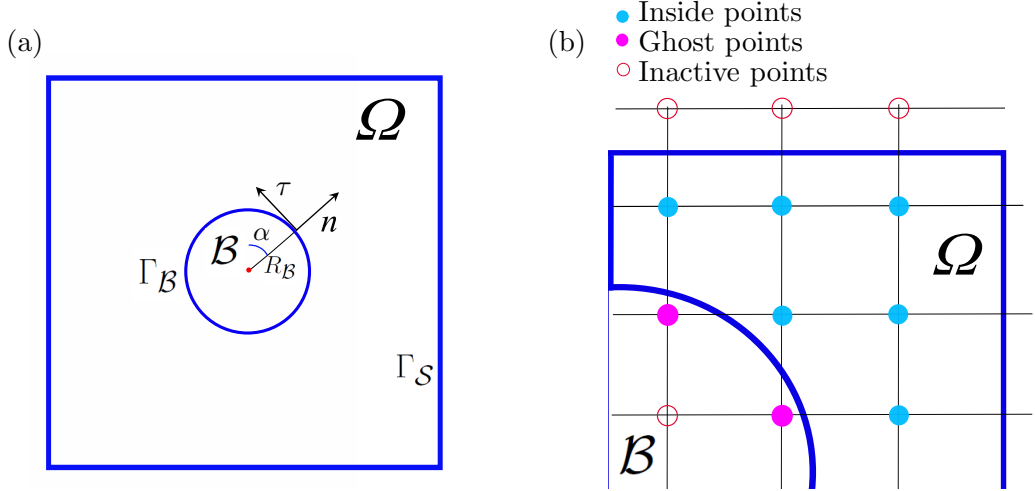
\begin{figure}
		\centering
		\begin{minipage}[b]
			{.49\textwidth}
			\centering
			\begin{overpic}[abs,width=0.75\textwidth,unit=1mm,scale=.25]{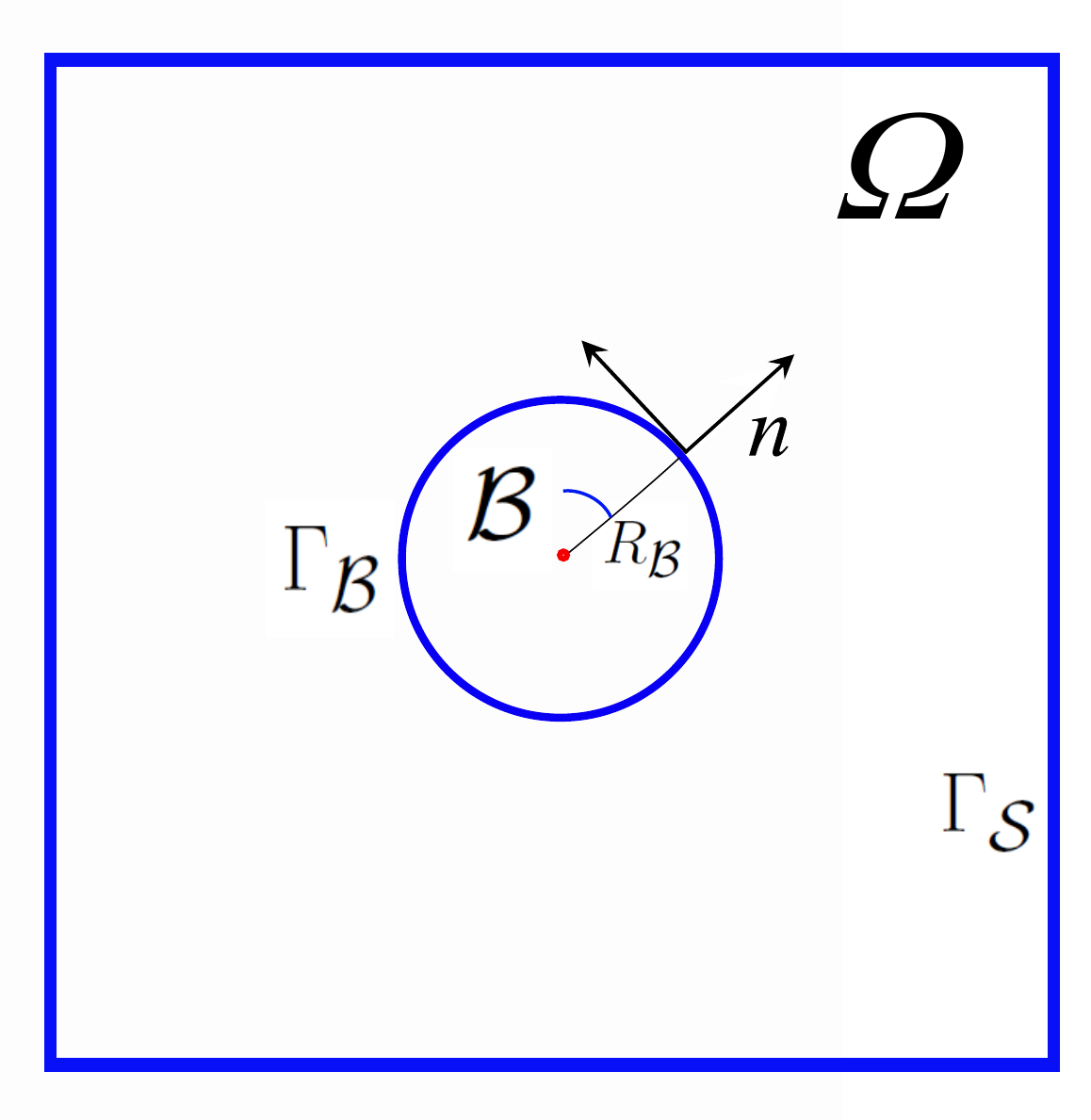}
				\put(-3.,58){(a)}
				\put(24.7,2.7){\dashline{3}(0,0)(0,44)}
				\put(29.5,33){\small $\alpha$}
				%\dashline(0,1){44}
				\put(32.5,39){$\tau$}
			\end{overpic}	%\includegraphics[width=0.7\textwidth]{domains3D_squared}
		\end{minipage}\hfill
		\begin{minipage}[b]
			{.49\textwidth}
			\centering
			\begin{overpic}[abs,width=0.8\textwidth,unit=1mm,scale=.25]{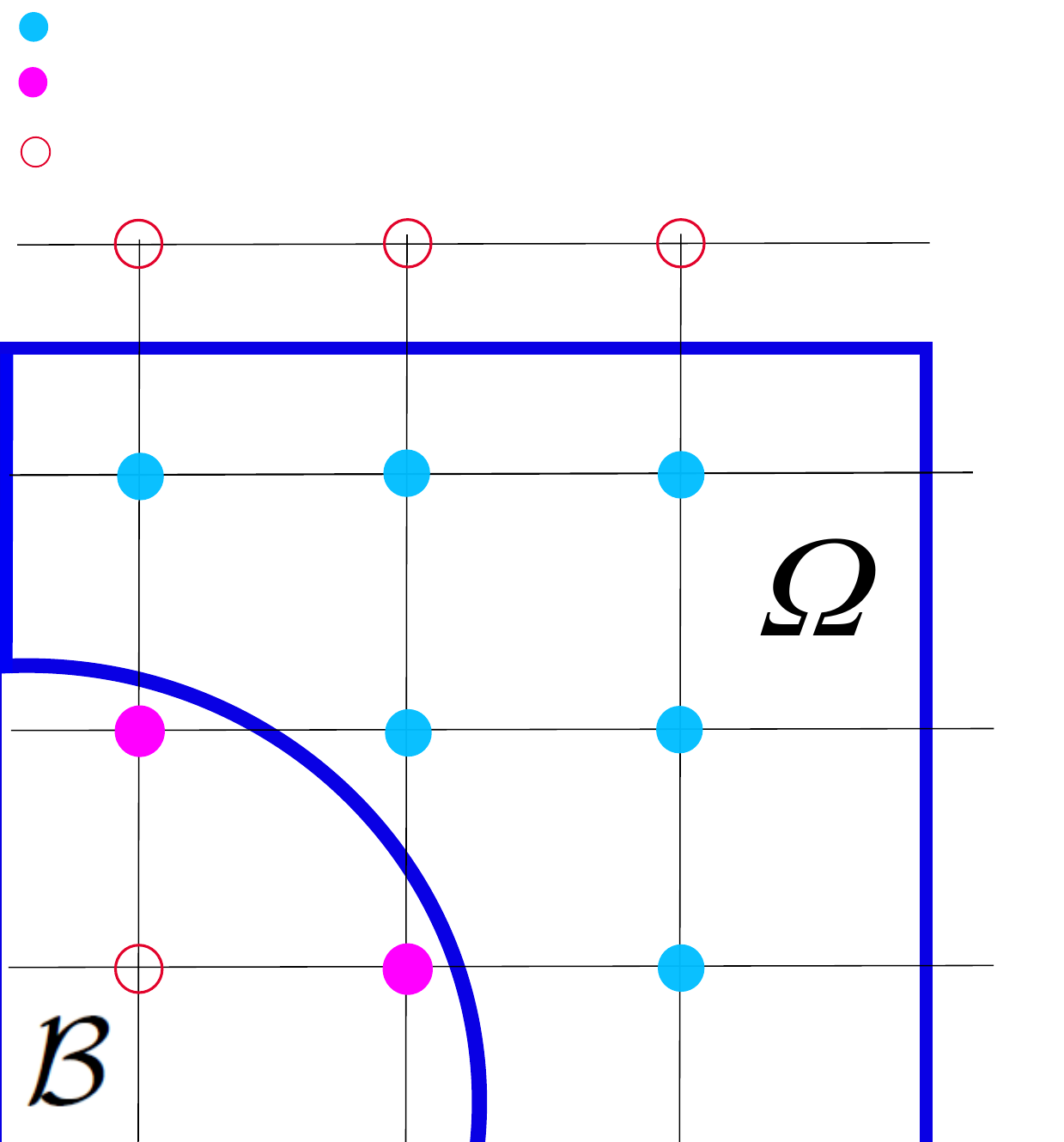}
				\put(4,54){{Inactive points}}
				\put(4,58){{Ghost points}}
				\put(4,62){{Inside points}}
				\put(-8.,58){(b)}
			\end{overpic}	%\includegraphics[width=0.75\textwidth]{classification_points3D}
		\end{minipage}
		\caption{\textit{(a): Representation of the domain $\Omega$, where $\Gamma_{\mathcal{S}}$ is the external wall, while $\mathcal{B} $ is the bubble with boundary $\Gamma_{\mathcal{B}}$ and radius $R_\mathcal{B}$. The dashed line represents the axis of specular symmetry in 2D and of rotation symmetry in 3D. (b): Classification of the inside (light blue circles), ghost (pink circles) and inactive (red hole circles) points.}}
		\label{2Ddomain_points}
	\end{figure}
	In this section, we describe the space discretization for Eqs.~(\ref{eq_ode}-\ref{eq_IC}). This numerical scheme in space provides a second order accurate method. For more details and accuracy tests about the coupling of the advection-diffusion equation with the Stokes equations see \cite{ASTUTO2023111880}, while, for a detailed description of the Navier-Stokes equations in arbitrary moving domains, see \cite{COCO2020109623}.

	The domain is $\Omega = ([-L_x/2,L_x/2]\times [-L_y/2,L_y/2])\setminus \mathcal{B}$, with $\mathcal{B}$ a circle centered in $(0,0)$ and radius $R_{\mathcal{B}}$ (see Fig.~\ref{2Ddomain_points}), and the problem reads:
	\begin{eqnarray*} 
		\left\{
		\begin{array}{l}
			\displaystyle     \frac{\partial c_\varepsilon}{\partial t} = D\Delta c_\varepsilon + \nabla \cdot (c_\varepsilon \vec{u}(t/\varepsilon))  \quad \text{ in $\Omega$} \\
			\displaystyle \nabla  c_\varepsilon \cdot n = 0 \quad  \text{ on $\Gamma_\mathcal{S}$}\\
			\displaystyle \mathcal{M}\frac{\partial  c_\varepsilon}{\partial t} = 
			\mathcal{M} D \Delta_\perp c_\varepsilon
			%        \frac{\partial ^2 c_\varepsilon}{\partial \tau ^2}
			-D\frac{\partial  c_\varepsilon}{\partial n }\quad \text{ on $\Gamma_\mathcal{B}$}
		\end{array}
		\right.
	\end{eqnarray*}
	where the expression for the velocity $\vec{u}(t/\varepsilon) = \vec{u}(x,y,t/\varepsilon)$ is known, $n$ is the outgoing normal vector to $\Gamma_\mathcal{B}$, and $\Delta_\perp = {\partial^2}/{\partial \tau^2}$ denotes the Laplace-Beltrami operator on the circumference of the circle
	(see Fig.~\ref{2Ddomain_points})\footnote{This is the same operator that we adopted in  \cite{ASTUTO2023111880} when working in the 3D axisymmetric case. We remark here that the exact operator in that case is  
		\[
		\Delta_\perp = \frac{1}{\sin \alpha} \frac{\partial}{\partial \tau} \left(\sin \alpha \frac{\partial}{\partial \tau}\right)
		\]
		where $\alpha$ is displayed in Fig.~\ref{2Ddomain_points} (a), and the dashed line represents the axis of specular symmetry in 2D and rotation symmetry in 3D.}  
	\subsection{Space discretization}
	We use a uniform square Cartesian cell-centered discretization, with $\Delta x = \Delta y = h$, and the set of grid points is $\mathcal{S}_h = (x_h,y_h) = \{(x_i,y_j)=(- h/2 + ih,- h/2 + jh), (i,j) \in \{1,\cdots,N\}^2 \}$, where $N \in \mathbb{N}$ and 
	$h = L_x/N$ with $L_x = L_y = 2$. Here we define the set of internal points $\Omega_h = \mathcal{S}_h \cap \Omega$, the set of bubble points $\mathcal{B}_h =\mathcal{S}_h \cap \mathcal{B}$ and the set of ghost points $\mathcal{G}_h$, which are bubble points with at least an internal point as neighbor, and are formally defines as follows
	\begin{equation}
		(x_i,y_j) \in \mathcal{G}_h \iff (x_i,y_j) \in \mathcal{B}_h \text{ and } \{(x_i \pm h,y_j),(x_i,y_j\pm h) \} \cap \Omega_h \neq \emptyset.
	\end{equation}
	The other grid points, $\mathcal{S}_h\setminus(\Omega_h \cup \mathcal{G}_h)$, are called inactive points. See Fig.~\ref{2Ddomain_points} (b) for a classification of inside, ghost, and inactive points.
	Let $N_I = |\Omega_h|$ and $N_G = |\mathcal{G}_h|$ be the cardinality of the sets $\Omega_h$ and $\mathcal{G}_h$, respectively, and $\mathcal{N} = N_I + N_G$ the total number of active points. 
	We compute the solution $ c_{\varepsilon,h} $ at the grid points of $\Omega_h \cup \mathcal{G}_h$, by using finite difference discretization of the equations on the $N_I$ internal grid points, and by suitable interpolation
	to define the conditions for the $N_G$ ghost values. Since each condition on a ghost point may involve other ghost points, the conditions on the ghost points are coupled. For this reason, the whole $\mathcal{N} \times \mathcal{N}$ system, with non-eliminated boundary conditions, is considered. 
	%\giovanni{Nel caso volessimo semplificare il modello numerico usando formalmente $N^2$ incognite, in modo da non impazzire con la numeriazione dei punti, si potrebbe aggiungere una frase del tipo: The matrix elements corresponding to the inactive points are diagonal, indicating that such points do not interact with the rest of the grid points.}
	%%% --------------------------------------- pec
	%The diffusion term suggests a central differencing scheme, which is second order accurate, and it is stable even in presence of an advection term, provided the so called \emph{mesh P\'{e}clet number} is smaller than four \cite{Wesseling2023600}. We therefore choose a spaces step $h$ such that $\left|{u}\right| < {4}/{h}$.
	%%% ------------------------------------------
	
	Representing $ c_\varepsilon, c^0_\varepsilon $ and $\vec{u}(t/\varepsilon)$ as column vectors $ c_{\varepsilon,h} = (\ldots,  c^\varepsilon_{i,j}, \ldots)^\top, c^0_{\varepsilon,h} = (\ldots,  c^{0,\varepsilon}_{i,j}, \ldots)^\top$, $\vec{u}_{h}(t/\varepsilon) = (\ldots,  \vec{u}_{i,j}, \ldots)^\top \in \mathbb{R}^{\mathcal{N}}$, where $\vec{u}_{i,j} = [{u}^x_{i,j},{u}^y_{i,j}]$, the problem \eqref{eq_ode} is then discretized in space, leading to a linear system
	\begin{eqnarray}
		\partial_t c_{\varepsilon,h} &=& \left(L_h +  Q_h\,\vec{u}_h(t/\varepsilon) \right)  c_{\varepsilon,h}, \quad c^0_{\varepsilon,h} = c_{\varepsilon,h}(t = 0) ,
		\label{pde2d}
	\end{eqnarray}
	where $L_h$ and $Q_h$ are $\mathcal{N} \times \mathcal{N}$ matrices representing the discretization of the derivative and the interpolation operators. 
	%We denote by $L_h^{(i,j)}= \left(L_h^{(i,j),1},\ldots,L_h^{(i,j),N_I+N_G} \right)$ and $Q_h^{(i,j)}= \left(Q_h^{(i,j),1},\ldots,Q_h^{(i,j),N_I+N_G} \right)$ the rows of $L_h$ and $Q_h$, respectively, associated with the grid point $(x_i,y_j)$. 
	If $P_{ij} = (x_i,y_j) \in \Omega_h$ is an internal grid point (as in Fig.~\ref{stencil} (a)), we discretize the diffusion and advection terms as follows
	\begin{eqnarray} 
		% L_h^{(i,j)}  c_{\varepsilon,h} &=& D \frac{ c^\varepsilon_{i+1,j}+ c^\varepsilon_{i-1,j}+ c^\varepsilon_{i,j+1}+ c^\varepsilon_{i,j-1} - 4 c^\varepsilon_{i,j}}{h^2} \\
		% Q_h^{(i,j)}  c_{\varepsilon,h} & = & \frac{ c^\varepsilon_{i+1,j} -  c^\varepsilon_{i-1,j} +  c^\varepsilon_{i,j+1} -  c^\varepsilon_{i,j-1}}{2h}.	
		\label{eq_Lh}
		L_h\, c_{\varepsilon,h}\Big|_{i,j} &=& D \left( \frac{c^\varepsilon_{i,j+1}  + c^\varepsilon_{i,j-1} + c^\varepsilon_{i+1,j}  + c^\varepsilon_{i-1,j} - 4c^\varepsilon_{i,j}}{h^2} \right)
		%D \frac{ c^\varepsilon_{i+1,j}+ c^\varepsilon_{i-1,j}+ c^\varepsilon_{i,j+1}+ c^\varepsilon_{i,j-1} - 4 c^\varepsilon_{i,j}}{h^2} 
		\\ \nonumber Q_h\,\left( \vec{u}_h(t/\varepsilon)\,c_{\varepsilon,h}\right)\Big|_{i,j} & = & \frac{u^x_{i+i,j}\,c^\varepsilon_{i+1,j} - u^x_{i-i,j}\,c^\varepsilon_{i-1,j} + u^y_{i,j+1}\,c^\varepsilon_{i,j+1} - u^y_{i,j-1}\,c^\varepsilon_{i,j-1}}{2h}
		%\frac{ c^\varepsilon_{i+1,j} -  c^\varepsilon_{i-1,j} +  c^\varepsilon_{i,j+1} -  c^\varepsilon_{i,j-1}}{2h}.	
	\end{eqnarray}
	To close the linear system, we must write an equation for each ghost point. If $G=(x_i,y_j) \in \mathcal{G}_h$ is a ghost point, then we discretize the boundary condition in \eqref{eq_bc_M_2D} for $\Gamma_{\mathcal{B}}$, following a ghost-point approach similar to the one proposed in \cite{COCO2013464,COCO2018299}, and summarised as follows. We first compute the closest boundary point $B \in \Gamma_\mathcal{B}$ by
	\[
	B =O + R_\mathcal{B} \frac{O-G}{|O-G|},
	\]
	where $O$ is the center of the bubble. Then, we identify the upwind nine-point stencil starting from $G=(x_G,y_G)=(x_i,y_j)$, containing $B=(x_B,y_B)$:
	\[
	\left\{ (x_{i+s_x m_x},x_{j+s_y m_y}) \colon m_x,m_y=0,1,2 \right\},
	\]
	where $s_x = \text{SGN} (x_B-x_G)$ and $s_y = \text{SGN} (y_B-y_G)$. The solution $ c_{\varepsilon,h} $ and its first and second derivatives are then interpolated at the boundary point $B$ using the discrete values $ c^\varepsilon_{i,j}$ on the nine-point stencil.
	%%%---------------
	% The interpolations can be obtained as tensor products of 1D interpolations in the axis directions. First, we interpolate in $r-$direction, moving from $G$ of a distance $\vartheta_x h$, and finding the three points in the rows $\{j, j+1, j+2\}$ (whole circle points in Fig.~\ref{stencil} (b)). Analogously, we move from $G$ in $y-$direction, of a distance $\vartheta_y h$, and we find the three points in the columns $\{i,i+1,i+2\}$. 
	% \todo[inline]{add more details about interpolation}
	% In detail, the 1D quadratic interpolations using the grid points $x_{i-2},x_{i-1},x_{i}$ to evaluate the function, its first derivative and the second derivative on $x_i - \vartheta h$ are given by
	% \[
	% \widehat{ c }(x_i + \vartheta\,h) =  \sum_{m=0}^2 \gamma_{m}(\vartheta) \,  c^\varepsilon_{i+m},
	% \quad
	% \widehat{ c }'(x_i + \vartheta\,h) =  \sum_{m=0}^2 \gamma'_{m}(\vartheta) \,  c^\varepsilon_{i+m},
	% \quad
	% \widehat{ c }''(x_i + \vartheta\,h) =  \sum_{m=0}^2 \gamma''_{m}(\vartheta)\, c^\varepsilon_{i+m},
	% \]
	% where
	% \[
	% \gamma(\vartheta) = \left( \frac{(1-\vartheta)(2-\vartheta)}{2}, \quad \vartheta (2-\vartheta), \quad \frac{\vartheta(\vartheta-1)}{2} \right)
	% \]
	% \[
	% \gamma'(\vartheta) = \frac{1}{h} \left( \frac{(2\vartheta-3)}{2}, \quad 2(1-\vartheta), \quad \frac{(2\vartheta-1)}{2} \right)
	% \]
	% \[
	% \gamma''(\vartheta) = \frac{1}{h^2} \left( 1, \quad -2, \quad 1 \right).
	% \]
	%---------------------------------
	We start defining (see Fig.~\ref{stencil} (b))
	\[ 
	\vartheta_x =  s_x (x_B-x_G)/h, \qquad
	\vartheta_y =  s_y (y_B-y_G)/h,
	\]
	with $0\leq \vartheta_x,\vartheta_y < 1$.
	The 2D interpolation formulas are:
	\begin{multline}\label{coeffsLSstencil} 
		\widehat{c}(B) = \sum_{m_x,m_y=0}^2 l_{m_x}(\vartheta_x) l_{m_y}(\vartheta_y)  c _{i+s_x m_x,j+s_y m_y},
		\\
		\frac{\partial \widehat{c}}{\partial x}(B) = s_x \sum_{m_x,m_y=0}^2 l'_{m_x}(\vartheta_x) l_{m_y}(\vartheta_y)  c _{i+s_x m_x,j+s_y m_y},
		%\\
		%\frac{\partial \widehat{c}}{\partial }(B) = s_y \sum_{m_x,m_y=0}^2 l_{m_x}(\vartheta_x) l'_{m_y}(\vartheta_y)  c _{i+s_x m_x,j+s_y m_y},
		\\
		\frac{\partial^2 \widehat{c}}{\partial x^2}(B) = \sum_{m_x,m_y=0}^2 l''_{m_x}(\vartheta_x) l_{m_y}(\vartheta_y)  c _{i+s_x m_x,j+s_y m_y},
		%\\
		%\frac{\partial^2 \widehat{c}}{\partial y^2}(B) = \sum_{m_x,m_y=0}^2 l_{m_x}(\vartheta_x) l''_{m_y}(\vartheta_y)  c _{i+s_x m_x,j+s_y m_y},
		\\ 
		\frac{\partial^2 \widehat{c}}{\partial x \partial y}(B) = s_x\, s_y \sum_{m_x,m_y=0}^2 l'_{m_x}(\vartheta_x) l'_{m_y}(\vartheta_y)  c _{i+s_x m_x,j+s_y m_y},
	\end{multline}
	where
	\[
	l(\vartheta_\alpha) = \left( \frac{(1-\vartheta_\alpha)(2-\vartheta_\alpha)}{2}, \, \vartheta_\alpha (2-\vartheta_\alpha), \, \frac{\vartheta_\alpha(\vartheta_\alpha-1)}{2} \right),
	\]
	\[
	l'(\vartheta_\alpha) = \frac{1}{h} \left( \frac{(2\vartheta_\alpha-3)}{2}, \, 2(1-\vartheta_\alpha), \, \frac{(2\vartheta_\alpha-1)}{2} \right), \, l''(\vartheta_\alpha) = \frac{1}{h^2} \left( 1, \, -2, \, 1 \right), \quad \alpha = x,y,
	\] 
	%\giovanni{Secondo me manca un divisione per $h$ nella espressione della derivata prima ed una divisione per $h^2$ nella espressione della derivata seconda, poiche la derivata di $\vartheta$ rispetto a $x$ è $1/h$}
	and where we omit ${\partial \widehat{c}}/{\partial y}(B)$ and ${\partial^2 \widehat{c}}/{\partial y^2}(B)$ they are analogue to the $x-$coordinate derivatives. Finally, the rows of $L_h$ associated with the ghost point $G=(x_G,y_G)$ are defined by evaluating the boundary condition on $\Gamma_\mathcal{B}$, i.e.
	\begin{equation}\label{QHghost}
		L_h c_{\varepsilon,h}\Big|_B = D \left. 
		\Delta_\perp \widehat{c}
		%\frac{\partial ^2 \widehat{ c }}{\partial \tau ^2} 
		\right|_B - 
		\frac{D}{\mathcal{M}} \left. \frac{\partial \widehat{ c }}{\partial n} \right|_B,
	\end{equation}
	and 
	\begin{equation}
		\label{normaleq}
		\frac{\partial }{\partial \tau} = \tau_x\,\frac{\partial }{\partial x} + \tau_y\,\frac{\partial }{\partial y}, \quad
		\frac{\partial }{\partial n} = n_x\,\frac{\partial }{\partial x} + n_y\,\frac{\partial }{\partial y},
	\end{equation}
	% \begin{eqnarray}\label{normaleq}
		% 	\frac{\partial }{\partial n} = n_x\,\frac{\partial }{\partial x} + n_y\,\frac{\partial }{\partial y}, &\qquad
		% 	\displaystyle \frac{\partial ^2}{\partial \tau ^2} =
		% 	\displaystyle \tau_x^2\,\frac{\partial^2 }{\partial x^2} + 2 \tau_x \tau_y\, \frac{\partial }{\partial x}\frac{\partial }{\partial y} + \tau_y^2\,\frac{\partial^2 }{\partial y^2}, %\\
		% 	% \qquad
		% 	%(n_x,n_y) = \frac{O-G}{|O-G|}, & \qquad (\tau_x,\tau_y)= (-n_y,n_x).
		% \end{eqnarray}
	with $(\tau_x,\tau_y)= (-n_y,n_x)$, $\cot \theta = n_x/n_y$. 
	For a spherical bubble, \\
	$(n_x,n_y) = ({O-G})/{|O-G|},\, \, \cot \theta =x/y$.
	%At the end, the problem \eqref{system3D} is discretized in space, leading to a linear system
	% \begin{equation}\label{eq_discr_space}
		% {\partial_t} c_{\varepsilon,h} = \left(L_h + Q_h \vec{u}_h(t/\varepsilon)\right) c_{\varepsilon,h}, \qquad c_{\varepsilon,h}(t=0) = c^0_{\varepsilon,h}
		% \end{equation}
	% where $L_h$ and $Q_h$ are the $\mathcal{N} \times \mathcal{N}$ matrices representing the discretization of the diffusive Laplacian operator and $Q_h$ the advection term in system~\eqref{system3D}.

\subsection{Complex-shaped bubbles: a level-set approach}
The discretization in the previous sections for the spherical bubble can extended to more general shapes by adopting a level-set approach. In detail, the bubble $\mathcal{B}$ can be implicitly defined by a level set function $\phi(x,y)$ that is positive inside the bubble, negative outside and zero on the boundary $\Gamma_\mathcal{B}$ (see, for example, \cite{sussman1994level,Osher,russo2000remark,book:72748}):
\begin{eqnarray}
	\mathcal{B} = \{(x,y): \phi(x,y) > 0\}, \qquad
	\Gamma_\mathcal{B} = \{(x,y): \phi(x,y) = 0\}.
\end{eqnarray}
The unit normal vector $n$ in \eqref{normaleq} can be computed as $n = \frac{\nabla \phi }{|\nabla \phi|}$
%\begin{equation}\label{LSnormal}
%	n = \frac{\nabla \phi }{|\nabla \phi|}
%\end{equation}
where the level-set function $\phi$ is assumed to be explicitly known. For a spherical bubble $\mathcal{B}$ centered at the origin, the most convenient level-set function in terms of numerical stability is the the signed distance function between $(x,y)$ and $\Gamma_\mathcal{B}$, i.e.\ $\phi=R_\mathcal{B}-\sqrt{x^2+y^2}$.\footnote{For arbitrary smooth level set function, and for small distances, a good approximation of the signed distance function $d(\vec{x},t)$ is given by
	$d(\vec{x},t) \approx \phi(\vec{x},t)/|\nabla\phi(\vec{x},t)|$.}

\begin{figure}[htp]
	\centering
	\hfill
	\begin{minipage}[b]
		{.45\textwidth}
		\centering
		\begin{overpic}[abs,width=0.75\textwidth,unit=1mm,scale=.25]{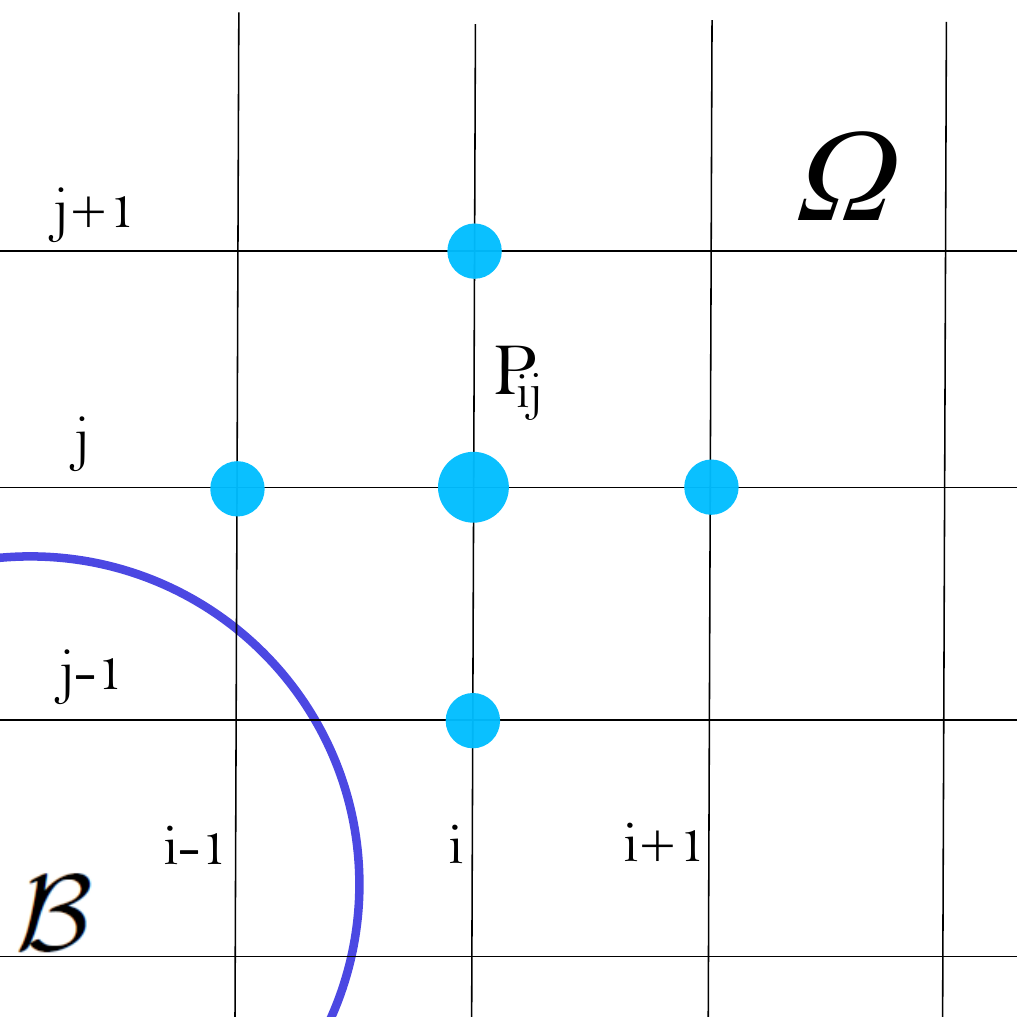}
			\put(-4.,50){(a)}
		\end{overpic}  
	\end{minipage}\hfill
	\begin{minipage}[b]
		{.49\textwidth}
		\centering
		\begin{overpic}[abs,width=0.65\textwidth,unit=1mm,scale=.25]{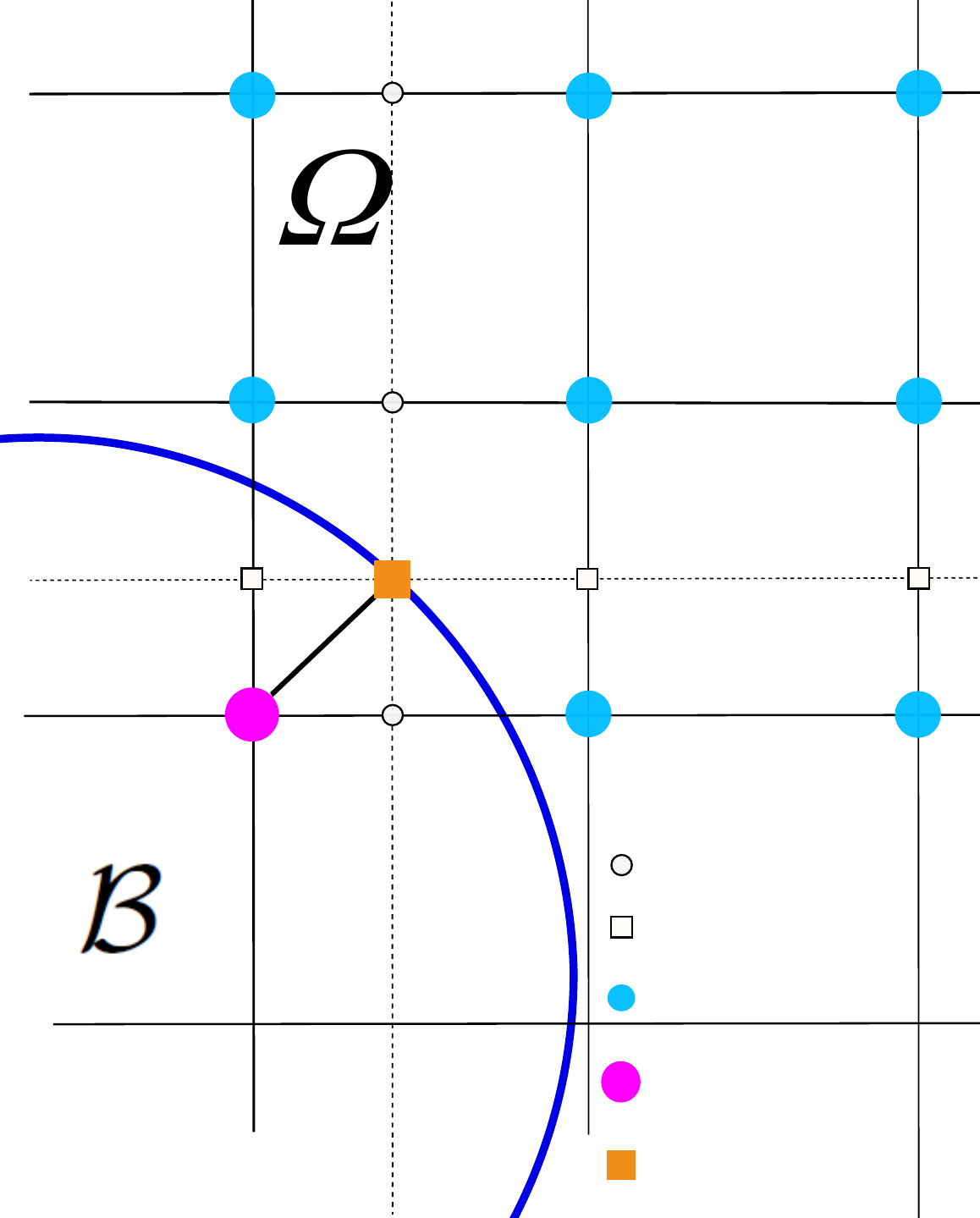}
			\put(12.5,19.5){\small $\vartheta_x h$}
			\put(5,27){\small $\vartheta_y h$}
			\put(-7.,53){(b)}
			\put(32,17){\footnotesize $x$-interpolation}
			\put(32,14){\footnotesize  $y$-interpolation}
			\put(32,10){\footnotesize  Inside points}
			\put(32,6){\footnotesize  Ghost points}
			\put(32,1.5){\footnotesize  Boundary points}
			\put(21,32){B}
			\put(8.5,19){G}
		\end{overpic}  
	\end{minipage}
	\hspace*{\fill}
	\caption{\textit{(a): Five-point stencil for the discrete operator $L_h$ of the second derivatives in space for the internal points $P_{ij} = (x_i,y_j)$. (b): Representation of the Upwind nine-points stencil associated with the ghost point $G$, boundary point $B$ and the relative outgoing normal vector $n$ to $\Gamma$.}}
	\label{stencil}
\end{figure}
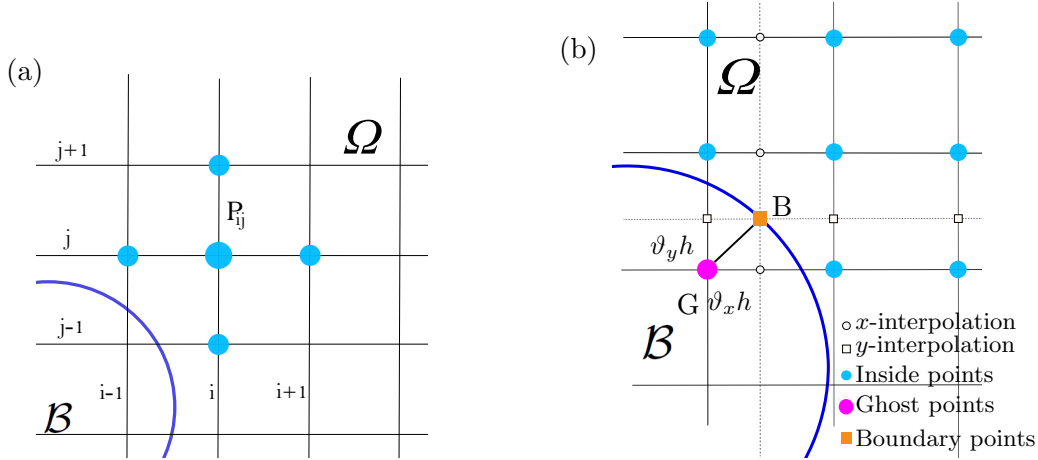

\section{A two--scale formulation}
In this section our goal is to provide a formal derivation of the asymptotic models for Eq.~\eqref{pde2d}, to $1^{\rm st}$ and $2^{\rm nd}$ order in $\varepsilon$, and how the two--scale formulation can be used to derive uniformly accurate numerical schemes. The strategy is to formulate and analyze the equations obtained by decoupling the slow variable $t$ and the fast variable $\Theta = t/\varepsilon$, following the same strategy adopted in  
\cite{chartier2020new,chartier2022derivative,crouseilles2013asymptotic,crouseilles2017uniformly}. In this way, we obtain the two--scale formulation of Eq.~\eqref{pde2d}, and then we derive the asymptotic models. Note that, one could also use the two--scale formulation to derive uniformly accurate numerical schemes solving the original equations in the vein of \cite{chartier2020new,chartier2022derivative,crouseilles2013asymptotic,crouseilles2017uniformly}. Since in this paper we are solving a preliminary case, in which the velocity is explicitly known, adding a new variable $\Theta$ is not necessary. We will follow another way to construct uniformly accurate schemes.  

The use of averaging methods is very common when multiple scales are involved, and the aim of the averaging is to weaken the stiffness of the short scale $t/\varepsilon$. %, and to choose a region in which the considered time average is large compared to small scale and small compared to the large one. 
We will consider a reformulation of the problem, that is more general and in which we add a degree of freedom when defining the new variable $\Theta$. The general problem is easier to solve numerically because its initial condition can be chosen in such a way that the solution, and its first and second time derivatives, are bounded in $\varepsilon$. %The strategy of the two--scale formulation is to derive the expression for its initial condition from the original initial condition of the model, that remains unaltered.

\subsection{Formal derivation of the asymptotic models: first and second order in {$\varepsilon$} }
\label{section_two_scales}
In this subsection we derive the asymptotic models for Eq.~\eqref{pde2d}, to $1^{\rm st}$ and $2^{\rm nd}$ order in $\varepsilon$.  Given $h > 0$, we consider the bounded linear operators $L_h$ and $Q_h$ on $\mathbb{R}^{\mathcal{N}}$, that were defined in Section~\ref{section_discr_space}. Then we introduce the two--scale vector-valued function $C_h(t,\Theta) \in \mathbb{R}^{\mathcal{N}}$, that is $1-$periodic in $\Theta$, and satisfies $C_h(t,t/\varepsilon) = c_{\varepsilon,h}(t)$, for all $t\geq 0$. Differentiating this quantity with respect to $t$, leads us to consider the following equation for $C_h(t,\Theta)$
\begin{equation}
	\label{eq_2scales_t_ext}
	\partial_t  C_h(t,\Theta) + \frac{1}{\varepsilon}\partial_\Theta C_h(t,\Theta) = \mathcal{F}_\Theta( C_h(t,\Theta)),  \qquad C_h(0,\Theta) = C^0_{h}(\Theta)
\end{equation}
where 
\begin{equation}
	\label{eq_F_cal}
	\mathcal{F}_\Theta( C_h(t,\Theta) ) = \left( L_h + Q_h\, \vec{u}_h(\Theta)\right) C_h(t,\Theta).
\end{equation}

%\todo[inline]{isn't this related to "well prepared initial data"? We may assume that we are beyond an initial transient, so we may assume the initial data is well prepared, in the sense that it admits an asymptotic expansion in epsilon. If so, this issue should be better explained.}

It is important to observe that, while the initial condition for the Eq.~\eqref{pde2d} is known, we do not have an explicit expression for $C^0_{h}(\Theta)$, we only know that $C^0_{h}(0) = c_{\varepsilon,h}^0$ where $c_{\varepsilon,h}^0$ is defined in Eq.~\eqref{pde2d}. This provides us a degree of freedom in the choice of $C^0_{h}(\Theta)$, which is in general made in order to get a smooth solution in $\varepsilon$ for the two--scale Eqs.~(\ref{eq_2scales_t_ext}-\ref{eq_F_cal}). Note that, any solution to Eqs.~(\ref{eq_2scales_t_ext}-\ref{eq_F_cal}), such that $C_h^0(0) = c_{\varepsilon,h}^0$, provides the solution of the original problem $c_{\varepsilon,h}(t) = C_h(t,t/\varepsilon)$.

%To find a good choice for the initial condition, we start considering the following decomposition of the solution $C_h(t)$ 
Here we perform the asymptotic analysis starting from the following decomposition of $C_h(t,\Theta)$
\begin{equation}
	\label{eq_expr_C}
	C_h(t,\Theta) = \overline{C}_h(t) + \widetilde C_h(t,\Theta),
\end{equation} 
with the choice of $\overline{C}_h(t)$ and $\widetilde C_h(t,\Theta)$
\begin{equation}
	\label{eq_Ctilde}
	\overline{C}_h(t) = \langle C(t,\bullet) \rangle
	\equiv\int_0^1 C_h(t,\Theta) d\Theta, \qquad   \langle \widetilde C_h(t,\bullet) \rangle = 0 \quad \forall t.
\end{equation} 

Now we search for the expression of $\widetilde C_h(t,\Theta)$. To do this, we substitute the expression \eqref{eq_expr_C} in Eq.~\eqref{eq_2scales_t_ext}, that becomes
\begin{eqnarray}
	\label{eq_2scales_t_ext_c_expr}
	\displaystyle \partial_t\overline{C}_h(t) + \partial_t\widetilde C_h(t,\Theta) + \frac{1}{\varepsilon}\partial_\Theta\widetilde{ C}_h(t,\Theta) &=&\mathcal{F}_\Theta(\overline{ C }_h(t)+\widetilde C_h(t,\Theta))
	%\\%		\displaystyle \partial_t\overline{ c } & =& <\mathcal{F}_\Theta( c )>
\end{eqnarray}
while the average of the Eq.~\eqref{eq_2scales_t_ext_c_expr} satisfies
\begin{equation}
	\label{eq_average_Cbar}
	\partial_t\overline{C}_h(t)  = \langle \mathcal{F}_\bullet(\overline{ C }_h(t)+\widetilde C_h(t,\bullet)) \rangle,
\end{equation}
with $ <\mathcal{F}_{\bullet}(\overline{C}_h(t) + \widetilde C_h(t,\bullet))> = \int_0^1 \mathcal{F}_{\Theta'}(\overline{C}_h(t) + \widetilde C_h(t,\Theta')) d\Theta'$.  

Now we subtract Eq.~\eqref{eq_average_Cbar} from Eq.~\eqref{eq_2scales_t_ext_c_expr} 
\begin{eqnarray*}
	\partial_t\widetilde C_h(t,\Theta) + \frac{1}{\varepsilon}\partial_\Theta\widetilde C_h(t,\Theta) & = & \mathcal{F}_\Theta(\overline{C}_h(t)+\widetilde C_h(t,\Theta))-\left<\mathcal{F}_\bullet(\overline{C}_h(t)+\widetilde C_h(t,\bullet)) \right>,
\end{eqnarray*}
and after integrating over the interval $[0,\Theta]$, we obtain 
% \displaystyle \partial_\Theta\widetilde{ c } &= &\varepsilon \left(\mathcal{F}_\Theta(\overline{ c }+\widetilde{ c })-\left<\mathcal{F}_\Theta(\overline{ c }+\widetilde{ c })\right>\right)-\varepsilon\partial_t\widetilde{ c }\\
\begin{eqnarray*}
	\widetilde C_h(t,\Theta) &= &\varepsilon \int_{0}^{\Theta}\left(\mathcal{F}_\sigma(\overline{C}_h(t) + \widetilde C_h(t,\sigma))-<\mathcal{F}_{\bullet}(\overline{C}_h(t) + \widetilde C_h(t,\bullet))> - \partial_t \widetilde C_h(t,\sigma) \right) d\sigma + {\rm k}(t),% \\
\end{eqnarray*}
where 
${\rm k}(t)$ is constant in $\Theta$, and
from Eq.~\eqref{eq_Ctilde}, we deduce that
\begin{eqnarray*}
	{\rm k}(t) & = & -\varepsilon\langle \int_{0}^{\Theta}\left(\mathcal{F}_\sigma(\overline{C}_h(t) + \widetilde C_h(t,\sigma))-\langle \mathcal{F}_{\bullet}(\overline{C}_h(t) + \widetilde C_h(t,\bullet))\rangle - \partial_t \widetilde C_h(t,\sigma)\right)d\sigma \rangle.
	%		\displaystyle \widetilde{ c }(\Theta) &=&\varepsilon \left(I-<\cdot>\right)\int_{0}^{\Theta}\left(\mathcal{F}_\sigma(\overline{ c })-<\mathcal{F}_\sigma(\overline{ c })>\right)d\sigma\\
	%		\displaystyle \widetilde{ c } &= &O(\varepsilon)
\end{eqnarray*}
At the end, the expression for $\widetilde C_h(t,\Theta)$ reads% , and we can say that is first order in $\varepsilon$. We also observe that, since its derivative in time is bounded and independent of $\varepsilon$, also $\partial_t \widetilde C_h(t,\Theta)$ is first order in $\varepsilon$, thus, we obtain the expression of $\widetilde C_h(t,\Theta)$
\begin{equation}
	\label{eq_Ctilde_2}
	\widetilde C_h(t,\Theta) = \varepsilon \left(I - \langle \cdot \rangle \right) \int_{0}^{\Theta}\left(\mathcal{F}_\sigma(\overline{C}_h(t)) - \langle\mathcal{F}_{\bullet}(\overline{C}_h(t) )\rangle \right)  d\sigma + O(\varepsilon^2),
\end{equation}
where we used that we formally have $\widetilde C_h(t,\Theta) = O(\varepsilon)$ and $\partial_t \widetilde C_h(t,\Theta) = O(\varepsilon)$, uniformly in $\varepsilon$.

At this point, the decomposition of the solution $C_h$ in Eq.~\eqref{eq_expr_C} can be rewritten as an expansion in $\varepsilon$, as follows
\begin{equation}
	\label{eq_Ch2}
	C_h = \overline{C}_h(t) + \varepsilon C_h^1(t,\Theta) + O(\varepsilon^2)
\end{equation}
where  
\begin{equation}
	\label{eq_C1h}
	C^1_h(t,\Theta) = \left(I - \langle \cdot \rangle \right) \int_{0}^{\Theta}\left(\mathcal{F}_\sigma(\overline{C}_h(t)) - \langle\mathcal{F}_{\bullet}(\overline{C}_h(t) )\rangle \right) d\sigma.
\end{equation}

Now we want to derive an approximate equation for $\overline{C}_h$, to $1^{\rm st}$ and $2^{\rm nd}$ order in $\varepsilon$, and we obtain it by substituting $C^1_h(t,\Theta) $ in Eq.~\eqref{eq_average_Cbar}
\begin{equation}
	\partial_t\overline{C}_h(t)  = \langle \mathcal{F}_\bullet(\overline{ C }_h(t)+\varepsilon  C^1_h(t,\bullet)) \rangle + O(\varepsilon^2).
\end{equation}
To close the system, we need to find a good choice for the initial condition, and we derive it from Eqs.~(\ref{eq_Ch2}-\ref{eq_C1h}). We choose 
\begin{equation}
	\label{eq_Ch0_theta}
	C_h^0(\Theta) = \overline{C}_h(0) + \varepsilon \left(I - \langle \cdot \rangle \right) \int_{0}^{\Theta}\left(\mathcal{F}_\sigma(\overline{C}_h(0)) - \langle\mathcal{F}_{\bullet}(\overline{C}_h(0) )\rangle \right) d\sigma,
\end{equation}
and if we evaluate Eq.~\eqref{eq_Ch0_theta} in $\Theta = 0$, with the assumption that $C_h^0(0) = c^0_{\varepsilon,h}$, we obtain
\begin{eqnarray}
	\nonumber
	\overline{C}_h(0) &=& c^0_{\varepsilon,h} + \varepsilon \langle \int_{0}^{\bullet}\left(\mathcal{F}_\sigma(\overline{C}_h(0)) - \langle\mathcal{F}_{\bullet}(\overline{C}_h(0) )\rangle \right) d\sigma \rangle \\ \label{eq_ic_Cbarh}
	&=& c^0_{\varepsilon,h} + \varepsilon \langle \int_{0}^{\bullet}\left(\mathcal{F}_\sigma(c^0_{\varepsilon,h}) - \langle\mathcal{F}_{\bullet}(c^0_{\varepsilon,h} )\rangle \right)d\sigma \rangle + O(\varepsilon^2).
\end{eqnarray}
At the end, we substitute Eq.~\eqref{eq_ic_Cbarh} in Eq.~\eqref{eq_Ch0_theta}, obtaining the initial condition $C_h^0(\Theta)$ as a function of the initial datum $c^0_{\varepsilon,h}$ of the  Eq.~\eqref{pde2d}
\begin{equation}
	C_h^0(\Theta) = c^0_{\varepsilon,h} + \varepsilon\int_0^\Theta \left(\mathcal{F}_\sigma(c^0_{\varepsilon,h}) - \langle\mathcal{F}_{\bullet}(c^0_{\varepsilon,h} )\rangle \right)d\sigma + O(\varepsilon^2).
\end{equation}

Finally, we have all the ingredients to write the asymptotic models, to $1^{\rm st}$ and $2^{\rm nd}$ order in $\varepsilon$:
\begin{itemize}
	\item[i)] The $1^{\rm st}$-order model reads
	\begin{eqnarray}
		\partial_t \overline{C}_h(t) &=& \langle \mathcal{F}_\bullet(\overline{C}_h(t)) \rangle \quad {\rm in }\, \Omega_h, \, \forall t > 0 \\
		\overline{C}_h(t = 0) &=& c^0_{\varepsilon,h}
	\end{eqnarray}
	where $\, \langle\mathcal{F}_\bullet(\overline{ C }_h)\rangle = \langle \left(L_h\, + Q_h\, \vec{u}_h(\bullet)\right) \overline{C}_h(t)\rangle, \,$
	and, from Eq.~\eqref{eq_Ch2}, $C_h = \overline{C}_h(t)$ is the truncation of the exact solution of Eq.~\eqref{eq_ode}, up to first order in $\varepsilon$. 
	\item[ii)] The $2^{\rm nd}$-order model reads
	\begin{eqnarray}
		\label{eq_2nd_Cbarh}
		\partial_t \overline{C}_h(t) &=& \langle\mathcal{F}_\bullet\left(\overline{ C }_h + \varepsilon C_h^1\right)\rangle \quad {\rm in }\, \Omega_h, \, \forall t \\
		\overline{C}_h(t = 0) &=& c^0_{\varepsilon,h} + \varepsilon\int_0^\Theta \left(\mathcal{F}_\sigma(c^0_{\varepsilon,h}) - \langle\mathcal{F}_{\bullet}(c^0_{\varepsilon,h} )\rangle \right)d\sigma
	\end{eqnarray}
	where $\, \langle\mathcal{F}_\bullet\left(\overline{ C }_h + \varepsilon C_h^1\right)\rangle = \langle \left(L_h\, + Q_h\, \vec{u}_h(\bullet)\right) \left(1 + \varepsilon Q_h \int_0^\bullet \vec{u}_h(\sigma)  d\sigma \right) \overline{C}_h(t)\rangle, \,$ \\
	and to obtain a second order approximation of the exact solution, we calculate 
	\begin{equation}
		\label{eq_Cbar_2nd}
		C_h(t,\Theta) = \left(1 + \varepsilon \left(I - \langle \cdot \rangle \right) \int_{0}^{\Theta}\left(\mathcal{F}_\sigma(\overline{C}_h(t)) - \langle\mathcal{F}_{\bullet}(\overline{C}_h(t) )\rangle \right) d\sigma\right)\, \overline{C}_h(t),
	\end{equation}
	where, an explicit expression of the RHS is obtained by substituting the expression of $\mathcal{F}_\Theta$ in Eq.~\eqref{eq_F_cal}
	\begin{eqnarray}
		C_h(t,\Theta) = \left(1 + \varepsilon Q_h \int_0^\Theta \vec{u}_h(\sigma)  d\sigma \right) \overline{C}_h(t).
	\end{eqnarray}
\end{itemize}
%\giovanni{Sono fuori sede, e non ho potuto controllare i passaggi matemastici della sezione precedente. Mi pare li avessi controllati la volta prima. Spero siano corretti!}
\subsection{A class of numerical schemes derived from the two--scale equation}
Here we show how to derive a uniformly accurate first and second order scheme for Eq.~\eqref{pde2d} with a two--scale formulation, though in this paper we choose another strategy. Our approach is based on a suitable time integral formulation and avoids the additional variable $\Theta = t/\varepsilon$. This variable could indeed increase the complexity of the numerical schemes. 

First we solve
\begin{eqnarray}
	\label{eq_1st_theta}
	\partial_t {C}_h(t,\Theta) + \frac{1}{\varepsilon}\partial_\Theta C_h(t,\Theta) &=&  \mathcal{F}_\Theta(C_h(t,\Theta)) \quad {\rm in }\, \Omega_h, \, \forall t > 0 \\ \label{eq_1st_theta_2}
	{C}_h(0,\Theta) &=& C^0_h(\Theta) = c^0_{\varepsilon,h}
\end{eqnarray}
where the solution $C_h(t,\Theta)$ of this problem, and its first time derivative, are uniformly bounded in $\varepsilon$. One can construct a first order uniformly accurate scheme for Eq.~\eqref{pde2d} starting from (\ref{eq_1st_theta}-\ref{eq_1st_theta_2}), and then setting $c_{\varepsilon,t}(t) = C_h(t,t/\varepsilon)$ (see, for instance, \cite{crouseilles2017uniformly}). %As we said before, since in our case the velocity is explicitly known, it is not convenient to add a new variable $\Theta$ to the system, increasing the complexity and the dimension of our initial problem.

%\giovanni{Non capisco: in pratica si approssima $C_h(t,\Theta)$ con una costante? Se non è così, come funziona il metodo, e perché non usiamo la tecnica illustrata nella sezione precedente? 
	%Forse potremmo eliminare questa sotto-sezione, e spiegarla meglio nel futuro articolo nel quale la usiamo per risolvere le equazioni di Stokes}

To obtain a second order scheme, we solve the following
\begin{eqnarray}
	\label{eq_2nd_theta}
	\partial_t {C}_h(t,\Theta) + \frac{1}{\varepsilon}\partial_\Theta C_h(t,\Theta) &=&  \mathcal{F}_\Theta(C_h(t,\Theta)) \quad {\rm in }\, \Omega_h, \, \forall t > 0 \\ \label{eq_2nd_theta_2}
	{C}_h(0,\Theta) &=& c^0_{\varepsilon,h} + \varepsilon\int_0^\Theta \left(\mathcal{F}_\sigma(c^0_{\varepsilon,h}) - \langle\mathcal{F}_{\bullet}(c^0_{\varepsilon,h} )\rangle \right)d\sigma
\end{eqnarray}
where the solution $C_h(t,\Theta)$ of this problem and its time derivatives, to first and second order, are uniformly bounded with respect to $\varepsilon$. Therefore, one can construct a second order uniformly accurate scheme for Eq.~\eqref{pde2d} by first solving (\ref{eq_2nd_theta}-\ref{eq_2nd_theta_2}), and then setting $c_{\varepsilon,t}(t) = C_h(t,t/\varepsilon)$. 

We do not apply this strategy here but we will use it in a more involved case where our model should be coupled with Stokes equations. This task is deferred to a work already under investigation.

	\section{First order accurate scheme, uniformly in {$\varepsilon$}}
	\label{section_1st_order}
	In this section we construct a scheme to solve Eq.~\eqref{pde2d}, that is first order accurate in time, and we prove that it is also uniformly accurate in $\varepsilon$. The only property that we use is that the first time derivative of the solution is uniformly bounded in $\varepsilon$, and this property is justified since we assume that the initial condition is not oscillatory in space and the boundary conditions {for the concentration} are not oscillatory in time. 
	
	In this section we define the operator $\mathcal{A}^1_{\Delta t}$ as a sum of three sparse matrices, the identity $I$ and the two tri-diagonal block matrices, $L_h$ and $Q_h$, defined in Eqs.~\eqref{eq_Lh}. In this way, it is easy to solve the linear system in Eq.~\eqref{eq_scheme_1st_2} with standard numerical techniques.

	\begin{pro} \label{pro_1st}
		Let $L_h$ and $Q_h$ be bounded operators in $\mathbb{R}^{\mathcal{N}}$ defined in Section~\ref{section_discr_space}, with $h>0$. Let $c_{\varepsilon,h}^0 \in \mathbb{R}^{\mathcal{N}}$ be bounded in $\varepsilon$, $\vec{u}_h \in \mathbb{R}^{\mathcal{N}}$ be a bounded given $1-$periodic function of time. Then, there exists a constant $\Delta t_0>0$ independent of $\varepsilon$ such that,  for all $\Delta t < \Delta t_0$, the following holds true:
		\begin{itemize}
			\item[i)] The operator $\mathcal{A}^1_{\Delta t} := I - \Delta t L_h - Q_h\int _{t^n}^{t^{n+1}}\vec{u}_h(s/\varepsilon)\,ds$ in $\mathbb{R}^\mathcal{N}$ is invertible.
			\item[ii)] The following:
			\begin{eqnarray} \label{eq_scheme_1st_1}
				c_{\varepsilon,h}^0 &=& c_{\varepsilon,h}(0),  \\ \label{eq_scheme_1st_2}
				c_{\varepsilon,h}^{n+1} &=& {\mathcal{A}^1_{\Delta t}}^{-1}c_{\varepsilon,h}^n,
			\end{eqnarray}
			for $n = 1,\cdots,M$ is a first order scheme, uniformly accurate with respect to $\varepsilon$, solving the Eq.~\eqref{pde2d}. In other words, we have $||c(t^n) - c^n||\leq K \Delta t$ for all $n = 1,\cdots,M$, with $K$ independent of $\varepsilon$, $\Delta t = t_{\rm fin}/M$, $t_{\rm fin}> 0$ and $t^n = n\Delta t$.    
		\end{itemize}
	\end{pro}
	
	\begin{proof}
		To prove i), we first show that  $\mathcal{A}^1_{\Delta t} \to I$, when $\Delta t \to 0$ since the operators $L_h$ and $Q_h$, and the function $\vec{u}_h$ are bounded. 
		% %%% REMOVE
		% REMOVE
		% and we prove that, $\mathcal{A}_{\Delta t}$ is invertible when $\Delta t \to 0$. We start calculating the norm of the operator $\mathcal{A}_{\Delta t}$, as follows
		% \[
		% ||\mathcal{A}_{\Delta t}|| = \left|\left|I - \Delta t L_h - Q_h\int _{t^n}^{t^{n+1}}\vec{u}_h(s/\varepsilon)\,ds\right|\right| \leq ||I|| + \Delta t (||L_h|| + ||Q_h||\,||u||_\infty),
		% \]
		% and we observe that $||\mathcal{A}_{\Delta t}||\to ||I||$ when $\Delta t \to 0$. Since $I$ is invertible, (i.e. ${\rm det}(I) \neq 0$) and the determinant is a continuous function, 
		Thus there exists a constant $\widehat K > 0$ (independent of $\varepsilon$), such that, $\forall \Delta t < \widehat K$,  ${\rm det}(\mathcal{A}^1_{\Delta t}) \neq 0$, and $\mathcal{A}^1_{\Delta t}$ is invertible.
		
		To prove ii), we first show how to deduce the scheme defined in Eqs.~(\ref{eq_scheme_1st_1}-\ref{eq_scheme_1st_2}). We start integrating  Eq.~\eqref{pde2d} between $t$ and $t^{n+1}$, with $[t,t^{n+1}] \subset [t^n, t_{\rm fin}]$ 
		\begin{eqnarray}
			c_{\varepsilon,h}(t^{n+1}) -  c_{\varepsilon,h}(t) &=& \int_{t}^{t^{n+1}}\left(L_h + Q_h \vec{u}_h(s/\varepsilon)\right) c_{\varepsilon,h}(s) ds.
			\label{eq_integral_form}
		\end{eqnarray} 
		%We observe that, in Eq.~\eqref{eq_integral_form} we choose to integrate between $t$ and $t^{n+1}$, instead of $t^n$ and $t$, with $t^n<t<t^{n+1}$, to avoid a positive sign in front of the operator $L_h^2$ in Eq.~\eqref{eq_2nd_order_method} (see Remark~\ref{remark_laplacian} in Section~\ref{section_2nd_order} for further details).
		Evaluating the quantity $ c_{\varepsilon,h}(t)$ in $t = t^n$, we obtain 
		\begin{eqnarray}
			\label{eq_subst_tn1}
			c_{\varepsilon,h}(t^{n+1}) -  c_{\varepsilon,h}(t^n) &= & \int_{t^n}^{t^{n+1}}\left(L_h + Q_h \vec{u}_h(s/\varepsilon) c_{\varepsilon,h}(s)\right)ds,
		\end{eqnarray} 
		and then the following
		\begin{eqnarray}
			\label{eq_first_appr}
			c_{\varepsilon,h}(t^{n+1}) -  c_{\varepsilon,h}(t^n) \approx  
			\left(
			\int_{t^n}^{t^{n+1}}\left(L_h + Q_h \vec{u}_h(s/\varepsilon)\right)\,ds\right)
			c_{\varepsilon,h}(t^{n+1}).
		\end{eqnarray} 
		%\giovanni{Ho messo la $c_{\varepsilon,h}$ a destra, visto che l'operatore si applica da sinistra e non da destra}
		Since the operators $L_h$ and $Q_h$ do not depend on time, Eq.~\eqref{eq_first_appr} becomes 
		\begin{equation}
			c_{\varepsilon,h}(t^{n+1}) -  c_{\varepsilon,h}(t^n) \approx L_h\,\Delta t  c_{\varepsilon,h}(t^{n+1}) + Q_h c_{\varepsilon,h}(t^{n+1})\int_{t^n}^{t^{n+1}}\vec{u}_h(s/\varepsilon)ds
			\label{1st order method}
		\end{equation} 
		where the integral $\int_{t^n}^{t^{n+1}}\vec{u}_h(s/\varepsilon)\,ds$ is calculated analytically.  At the end, the scheme reads
		\begin{equation}
			\label{eq_1st_ord_scheme}
			c_{\varepsilon,h}^{n+1} - c_{\varepsilon,h}^n = L_h\,\Delta t \,  c_{\varepsilon,h}^{n+1} + Q_h  c_{\varepsilon,h}^{n+1} \int _{t^n}^{t^{n+1}}\vec{u}_h(s/\varepsilon)\,ds
		\end{equation}
		where $c_{\varepsilon,h}^n \approx c_{\varepsilon,h}(t^n)$.
		
		Now we prove that the scheme in Eq.~\eqref{eq_1st_ord_scheme} is first order in $\Delta t$, uniformly in $\varepsilon$. %We start rewriting Eq.~\eqref{eq_integral_form} 
		% \begin{equation}
			%         c_{\varepsilon,h}(t^{n+1}) - c_{\varepsilon,h}(t^n) = \int_{t^n}^{t^{n+1}}\left( L_h + Q_h \vec{u}_h(s/\varepsilon)c_{\varepsilon,h}(s)\right)ds 
			% \end{equation}
		% and its relative numerical scheme in Eq.~\eqref{eq_1st_ord_scheme}
		% \begin{eqnarray}
			%     c_{\varepsilon,h}^{n+1} - c_{\varepsilon,h}^n &=& \int_{t^n}^{t^{n+1}} \left(L_h + Q_h \vec{u}_h(s/\varepsilon)\,ds \right) c_{\varepsilon,h}^{n+1}.
			% \end{eqnarray}
		We subtract Eq.~\eqref{eq_1st_ord_scheme} from Eq.~\eqref{eq_subst_tn1}, define the error at time $t^n$ as $e_n = c_{\varepsilon,h}(t^n) - c_{\varepsilon,h}^n$, and obtain
		\begin{equation}
			\label{eq_en_0}
			e_{n+1} - e_n =\int_{t^n}^{t^{n+1}} \left( L_h + Q_h \vec{u}_h(s/\varepsilon)\right)\left(c_{\varepsilon,h}(s) - c_{\varepsilon,h}^{n+1}\right)\,ds. 
		\end{equation}
		Here we substitute $\, c_{\varepsilon,h}(s) - c_{\varepsilon,h}^{n+1} = c_{\varepsilon,h}(s) -c_{\varepsilon,h}(t^{n+1})+c_{\varepsilon,h}(t^{n+1})- c_{\varepsilon,h}^{n+1}\,$ in Eq.~\eqref{eq_en_0}, obtaining
		\begin{eqnarray}
			\label{eq_en_1} \nonumber
			e_{n+1} - e_n & = & \displaystyle \int_{t^n}^{t^{n+1}} \left( L_h + Q_h \vec{u}_h(s/\varepsilon)\right)\left(c_{\varepsilon,h}(s) - c_{\varepsilon,h}(t^{n+1})\right)\,ds \\ && \displaystyle + \int_{t^n}^{t^{n+1}} \left( L_h + Q_h \vec{u}_h(s/\varepsilon)\right)\,ds \left(c_{\varepsilon,h}(t^{n+1}) - c_{\varepsilon,h}^{n+1}\right).
		\end{eqnarray}
		%\giovanni{Per piacere, cerchiamo di essere consistenti nell'uso della freccia: $u$ o $\vec{u}$? Scegliamo un simbolo per il vettore velocità ed usiamolo in maniera consistente in tutto l'articolo.}
		Since $\vec{u}_h(t/\varepsilon)$ and $c_{\varepsilon,h}(t)$ are bounded, also the first derivative in time of $c_{\varepsilon,h}(t)$ is bounded, i.e.,  there exists a constant $K_1$, independent of $\varepsilon$, such that
		\[ ||\partial_t  c_{\varepsilon,h} ||_{\mathcal{L}^\infty(\Omega)} = ||\left(L_h + Q_h \vec{u}_h(t/\varepsilon)\right) c _{\varepsilon,h}||_{\mathcal{L}^\infty(\Omega)} < K_1.
		\]
		where $K_1 = \left(||L_h|| + ||Q_h||\,||u||_\infty\right) ||c_{\varepsilon,h}||_{\mathcal{L}^\infty(\Omega)}$. From this estimate, we obtain 
		\begin{equation}
			||c_{\varepsilon,h}(s) - c_{\varepsilon,h}(t^{n+1})||\leq K_1 ||s - t^{n+1}|| \leq K_1 \Delta t, \qquad \forall s \in [t^n,t^{n+1}].
		\end{equation}
		Now we consider the first term of the RHS of Eq.~\eqref{eq_en_1}, that is bounded by
		\begin{equation}
			\left|\left|\int_{t^n}^{t^{n+1}} \left( L_h + Q_h \vec{u}_h(s/\varepsilon)\right)\left(c_{\varepsilon,h}(s) - c_{\varepsilon,h}(t^{n+1})\right)\,ds \right|\right| \leq K_1 K_2 \Delta t^2
		\end{equation}
		where $K_2 = ||L_h + Q_h \vec{u}_h(s/\varepsilon)||_{\mathcal{L}^\infty(\Omega)}$, and, it does not depend on $\varepsilon$. 
		
		Considering now the norm of the Eq.~\eqref{eq_en_1}, the following inequalities hold
		\begin{equation}
			\label{eq_en}
			||e_{n+1}|| - ||e_n|| \leq || e_{n+1} - e_n || \leq K_1 K_2\Delta t^2 + ||e_{n+1}|| K_2\Delta t \leq K_3\Delta t^2 + ||e_{n+1}||K_3\Delta t,
		\end{equation}
		where $K_3 = \max\{K_1K_2,K_2\}$. After some algebra, adding $\Delta t$ in both sides, and defining $E_n = ||e_n||$, we have
		\begin{equation}
			\label{eq_En_rec}
			E_{n+1} + \Delta t \leq (1-K_3 \Delta t)^{-1}(E_n + \Delta t).
		\end{equation}
		if, and only if, $\Delta t < 1/K_3$. From Eq.~\eqref{eq_En_rec}, recursively to $E_0 = 0$, we obtain that 
		\[ \displaystyle E_n + \Delta t \leq (1-K_3\Delta t)^{-n}\Delta t \leq \left(1-K_3\frac{t_{\rm fin}}{N} \right)^{-N}\Delta t % \to \Delta t e^{K_3T}
		\]
		which follows that, there exists a constant $K_4$, such that
		\[E_n \leq \left(\left(1-K_3\frac{t_{\rm fin}}{N} \right)^{-N} -1 \right)\Delta t \leq K_4\Delta t\]
		$K_4$ independent of $N$, since $ \lim_{N\to +\infty}\left(\left(1-K_3\,{t_{\rm fin}}/{N} \right)^{-N} -1 \right) = \exp{(K_3t_{\rm fin})} - 1$. \\
		To conclude the proof, we define $\Delta t_0 = \min\{\widehat K, 1/K_3\}.$
	\end{proof}
	\section{Second order accurate scheme, uniformly in {$\varepsilon$}}
	\label{section_2nd_order}
	In this section we construct a scheme to solve Eq.~\eqref{pde2d}, that is second order accurate in time, and we prove that is also uniformly accurate in $\varepsilon$. To do this, we follow the approach showed in Section~\ref{section_1st_order} and the only property that we use is that the first time derivative of the solution is uniformly bounded in $\varepsilon$, justified from the assumption that the initial condition is not oscillatory in space and the boundary conditions are not oscillatory in time. %So, in order to obtain a second order uniformly accurate scheme, we only need that the first time derivative is uniformly bounded in $\varepsilon$.
	We also show how to deal with the integration of the Laplacian operator in the proof of Proposition~\ref{pro_2nd}. In Remark~\ref{remark_laplacian} we also show other alternatives.
	
	%%%%% -------
	
	\begin{pro} \label{pro_2nd}
		Let $L_h$ and $Q_h$ be finite dimensional bounded operators in $\mathbb{R}^{\mathcal{N}}$ defined in Section~\ref{section_discr_space}, with $h>0$. Let $c_{\varepsilon,h}^0 \in \mathbb{R}^{\mathcal{N}}$ be bounded in $\varepsilon$, and let $\vec{u}_h \in \mathbb{R}^{\mathcal{N}}$ be a bounded given $1-$periodic function in time. Then, there exists a constant $\Delta t_0>0$ independent of $\varepsilon$, such that for all $\Delta t < \Delta t_0$, the following holds true:
		\begin{itemize}
			\item[i)] The operator $\mathcal{A}^2_{\Delta t} = I - \mathbb{M}$ in $\mathbb{R}^{\mathcal{N}}$ is invertible, where
			\begin{align} \nonumber & \mathbb{M} = L_h\,\Delta t  + Q_h\int _{t^n}^{t^{n+1}}\vec{u}_h(s/\varepsilon)\,ds- \frac{1}{2}L_h^2\Delta t^2 \\ \nonumber & - L_hQ_h\int  _{t^n}^{t^{n+1}} \int_{s}^{t^{n+1}}\vec{u}_h(\sigma/\varepsilon)\,d\sigma ds  - Q_hL_h\int_{t^n}^{t^{n+1}}(t^{n+1} - s)\vec{u}_h(s/\varepsilon)\,ds \\ \nonumber &- Q_h^2 \int_{t^n}^{t^{n+1}}\vec{u}_h(s/\varepsilon)\int_{s}^{t^{n+1}}\vec{u}_h(\sigma/\varepsilon)\,d\sigma\,ds, 
			\end{align}    
			\item[ii)] The following scheme:
			\begin{eqnarray} \label{eq_scheme_2nd_1}
				c_{\varepsilon,h}^0 &=& c_{\varepsilon,h}(0),  \\ \label{eq_scheme_2nd_2}
				c_{\varepsilon,h}^{n+1} &=& {\mathcal{A}^2_{\Delta t}}^{-1}c_{\varepsilon,h}^n, \qquad {\rm for } \quad n = 1,\cdots,M
			\end{eqnarray}
			is a second order scheme, uniformly accurate with respect to $\varepsilon$, solving the Eq.~\eqref{pde2d}. In other words we have $||c(t^n) - c^n||\leq K \Delta t^2$,  for all $n = 1,\cdots,M$, with $K$ independent of $\varepsilon$, $t^n = n\Delta t$ and $\Delta t = t_{\rm fin}/M$.  
		\end{itemize}
	\end{pro}

	\begin{proof}
		To prove i), we prove that $\mathcal{A}^2_{\Delta t}$ is invertible when $\Delta t \to 0$. We start considering the norm of the operator $\mathcal{A}^2_{\Delta t}$ as before, and, analogously, since the operators $L_h$ and $Q_h$ are bounded in $\mathbb{R}^{\mathcal{N}}$, and the vector-valued function $\vec{u}_h$ is also bounded, it follows that
		$||\mathcal{A}^2_{\Delta t}||\to ||I||$ when $\Delta t \to 0$. Thus, following the same procedure as before, we say that there exists a constant $\widehat K > 0$, such that, $\forall \Delta t < \widehat K$, the operator $\mathcal{A}^2_{\Delta t}$ is invertible.
		
		To prove ii), we first show how to deduce the scheme defined in Eqs.~(\ref{eq_scheme_2nd_1}-\ref{eq_scheme_2nd_2}). First, we substitute $ c_{\varepsilon,h}(t)$ with $c_{\varepsilon,h}(s)$ in Eq.~\eqref{eq_integral_form}, and it can be rewritten as
		\begin{eqnarray}
			\label{eq_integral_form_s}
			c_{\varepsilon,h}(s) = c_{\varepsilon,h}(t^{n+1}) - 
			\int_{s}^{t^{n+1}}\left(L_h + Q_h \vec{u}_h(\sigma/\varepsilon)\right) c_{\varepsilon,h}(\sigma) d\sigma,
		\end{eqnarray}
		where we remind that the choice of integrating between $t$ and $t^{n+1}$, with $t<t^{n+1}$, is our strategy to obtain a negative sign in front of the operator $L_h^2$ in Eq.~\eqref{eq_2nd_order_method} (see Remark~\ref{remark_laplacian} for more details).
		
		Now we substitute Eq.~\eqref{eq_integral_form_s} in Eq.~\eqref{eq_integral_form}, as follows
		\begin{eqnarray} \nonumber
			&& c_{\varepsilon,h}(t^{n+1}) -  c_{\varepsilon,h}(t) = \int_{t}^{t^{n+1}}\left(L_h + Q_h \vec{u}_h(s/\varepsilon)\right) c_{\varepsilon,h}(s) ds\\ \nonumber
			&& = \int_{t}^{t^{n+1}}\left(L_h + Q_h \vec{u}_h(s/\varepsilon)\right)\left( c_{\varepsilon,h}(t^{n+1})-\int_{s}^{t^{n+1}}\left(L_h + Q_h \vec{u}_h(\sigma/\varepsilon)\right) c_{\varepsilon,h}(\sigma)d\sigma\right)ds,
			%\label{2nd_order_def}
		\end{eqnarray} 
		we evaluate the quantity $ c_{\varepsilon,h}(t)$ in $t=t^n$
		\begin{align}
			\label{eq_subst_tn}
			c_{\varepsilon,h}(t^{n+1}) -  c_{\varepsilon,h}(t^n) & = \\ \nonumber\displaystyle \int_{t^n}^{t^{n+1}} \left(L_h +  Q_h \vec{u}_h(s/\varepsilon)\right) & \left( c_{\varepsilon,h}(t^{n+1})-\int_{s}^{t^{n+1}}\left(L_h + Q_h \vec{u}_h(\sigma/\varepsilon)\right)c_{\varepsilon,h}(\sigma)d\sigma\,\right)ds
		\end{align} 
		and we approximate the equation in the following way
		\begin{align} \label{approx_sigma}
			\displaystyle c_{\varepsilon,h}^{n+1} -  c_{\varepsilon,h}^n & = \\ \nonumber \displaystyle \int_{t^n}^{t^{n+1}} & \left(L_h + Q_h \vec{u}_h(s/\varepsilon)\right) \left( c_{\varepsilon,h}^{n+1}-\int_{s}^{t^{n+1}}\left(L_h + Q_h \vec{u}_h(\sigma/\varepsilon)\right)d\sigma\, c_{\varepsilon,h}^{n+1}\right)ds,
		\end{align} 
		where $c^n_{\varepsilon,h} \approx c_{\varepsilon,h}(t^n)$. Going on with the computations of the integrals, we have
		\begin{eqnarray} \label{eq_2nd_order_method}
			\displaystyle
			&& c_{\varepsilon,h}^{n+1} -  c_{\varepsilon,h}^n %\int_{t^n}^{t^{n+1}}\left(L_h + Q_h \vec{u}_h(s)\right)ds c_{\varepsilon,h}^{n+1} -\int_{t^n}^{t^{n+1}}\left(L_h + Q_h \vec{u}_h(s)\right)\left(\int_{s}^{t^{n+1}}\left(L_h + Q_h \vec{u}_h(\sigma)\right)d\sigma\right)ds c_{\varepsilon,h}^{n+1}&\\
			%& = L_h\,\Delta t \,  c_{\varepsilon,h}^{n+1} + Q_h\int _{t^n}^{t^{n+1}}\vec{u}_h(s)\,ds\, c_{\varepsilon,h}^{n+1} -\int_{t^n}^{t^{n+1}}\left(L_h + Q_h \vec{u}_h(s)\right)\left(A(t^{n+1}-s) + Q_h\int _{s}^{t^{n+1}}\vec{u}_h(\sigma)\,d\sigma\right)ds c_{\varepsilon,h}^{n+1}&\\
			=\displaystyle \Delta t \, L_h\,  c_{\varepsilon,h}^{n+1} + Q_h\, c_{\varepsilon,h}^{n+1}\int _{t^n}^{t^{n+1}}\vec{u}_h(s/\varepsilon)\,ds - \frac{1}{2}\Delta t^2\,L_h^2 c_{\varepsilon,h}^{n+1} \\ \nonumber &&\displaystyle - L_hQ_h\,c_{\varepsilon,h}^{n+1}\int_{t^n}^{t^{n+1}} \int_{s}^{t^{n+1}}\vec{u}_h(\sigma/\varepsilon)\,d\sigma ds -Q_hL_h\, c_{\varepsilon,h}^{n+1}\int_{t^n}^{t^{n+1}} ( t^{n+1} - s)\vec{u}_h(s/\varepsilon)\,ds \\ \nonumber &&\displaystyle - Q_h^2\,c_{\varepsilon,h}^{n+1} \int_{t^n}^{t^{n+1}}\vec{u}_h(s/\varepsilon)\int_{s}^{t^{n+1}}\vec{u}_h(\sigma/\varepsilon)\,d\sigma\,ds.
		\end{eqnarray} 
		
		At this point, we show that the numerical scheme is second order accurate, uniformly in $\varepsilon$. Here we subtract Eq.~\eqref{approx_sigma} from Eq.~\eqref{eq_subst_tn}, as before, obtaining 
		\begin{eqnarray} \label{eq_en_2nd_order}
			& e_{n+1} - e_n = \displaystyle \int_{t^n}^{t^{n+1}} \left(L_h + Q_h  \vec{u}_h(s/\varepsilon)\right) e_{n+1} ds \\ \nonumber
			& \displaystyle - \int_{t^n}^{t^{n+1}} \left(L_h + Q_h  \vec{u}_h(s/\varepsilon)\right) \int_{s}^{t^{n+1}}\left(L_h + Q_h  \vec{u}_h(\sigma/\varepsilon)\right)(c_{\varepsilon,h}(\sigma) - c_{\varepsilon,h}^{n+1})d\sigma\, ds.
		\end{eqnarray}
		where, again, $e_n = c_{\varepsilon,h}(t^n) - c_{\varepsilon,h}^n$. Now we add and subtract the same quantity to the RHS as before, $\, c_{\varepsilon,h}(s) - c_{\varepsilon,h}^{n+1} = c_{\varepsilon,h}(s) -c_{\varepsilon,h}(t^{n+1})+c_{\varepsilon,h}(t^{n+1})- c_{\varepsilon,h}^{n+1} \,$
		%\[ \int_{t^n}^{t^{n+1}} \left(L_h + Q_h  \vec{u}_h(s/\varepsilon)\right) \int_{s}^{t^{n+1}}\left(L_h + Q_h  \vec{u}_h(\sigma/\varepsilon)\right)ds\,c_{\varepsilon,h}(t^{n+1})
		%\]
		%obtaining
		% \begin{align} \label{eq_en_2nd}
			%     &e_{n+1} - e_n =  \int_{t^n}^{t^{n+1}} \left(L_h + Q_h  \vec{u}_h(s/\varepsilon)\right) e_{n+1} ds \\  & \nonumber
			%     - \int_{t^n}^{t^{n+1}} \left(L_h + Q_h  \vec{u}_h(s/\varepsilon)\right) \int_{s}^{t^{n+1}}\left(L_h + Q_h  \vec{u}_h(\sigma/\varepsilon)\right)(c_{\varepsilon,h}(\sigma) - c_{\varepsilon,h}(t^{n+1}))d\sigma\, ds\\ & \nonumber -
			%     \int_{t^n}^{t^{n+1}} \left(L_h + Q_h  \vec{u}_h(s/\varepsilon)\right) \int_{s}^{t^{n+1}}\left(L_h + Q_h  \vec{u}_h(\sigma/\varepsilon)\right)d\sigma\, ds (c_{\varepsilon,h}(t^{n+1}) - c_{\varepsilon,h}^{n+1}). 
			% \end{align}
		and, using the analogue procedure and the same constants seen in Eq.~\eqref{eq_en}, Eq.~\eqref{eq_en_2nd_order} becomes
		\begin{equation*}
			( 1 - K_3\Delta t - K_3^2\Delta t^2) E_{n+1} - K_3\Delta t^3 \leq E_n  
		\end{equation*}
		where $E_n = ||e_n||$. 
		
		In this case, we add ${\Delta t^2}/({K_3\Delta t + 1})$ to both sides, and, after some algebra, the inequality reads
		\[
		E_{n+1} + \frac{\Delta t^2}{K_3\Delta t + 1} \leq ( 1 - K_3\Delta t - K_3^2\Delta t^2)^{-1} \left(E_{n} + \frac{\Delta t^2}{K_3\Delta t + 1} \right),
		\]
		if $K_3\Delta t(1 + K_3\Delta t) < 1$. Recursively we obtain (using $E_0 = 0$) 
		\[
		E_{n} + \frac{\Delta t^2}{K\Delta t + 1} \leq ( 1 - K_3\Delta t - K_3^2\Delta t^2)^{-n}  \frac{\Delta t^2}{K_3\Delta t + 1}.
		\]
		As before, it is easy to show that there exists a constant $\widehat K$ that does not depend on $N$, such that $E_n \leq \widehat{K} \Delta t^2$. To conclude the proof, we define $\Delta t_0 = \min\{\widehat K, \widetilde K\}$, where $\widetilde K$ is the positive solution of the equation $1 - K_3\Delta t - K_3^2\Delta t^2 = 0$.
	\end{proof}
	\begin{remark}
		\label{remark_laplacian}
		We noticed that, the choice of the time interval for the integration of Eq.~\eqref{pde2d} can be trivial, because the wrong integration can bring instability to the numerical scheme. We need a negative sign in front of the term $\left(L_h\,\Delta t  \right)^2/2$ in Eq.~\eqref{eq_2nd_order_method}, since the operator $L_h$ represents the discrete Laplacian in space. Here we show some examples in which some instability can be generated.
		
		Let us consider the integration of Eq.~\eqref{pde2d} between $t^n$ and $t$, and between $t$ and $t^{n+1}$, with $t^n<t<t^{n+1}$
		\begin{eqnarray}
			\label{eq_remark_1}
			c_{\varepsilon,h}(t) & = & c_{\varepsilon,h}(t^n) + \int_{t^n}^{t}\left(L_h + Q_h \vec{u}_h(s/\varepsilon)\right) c_{\varepsilon,h}(s) ds \\ \label{eq_remark_2}
			c_{\varepsilon,h}(t) &=& c_{\varepsilon,h}(t^{n+1}) - 
			\int_{t}^{t^{n+1}}\left(L_h + Q_h \vec{u}_h(s/\varepsilon)\right) c_{\varepsilon,h}(s) ds.
		\end{eqnarray}    
		The first possibility is that we iterate Eq.~\eqref{eq_remark_1} in itself,  but in this case it is obvious that a negative sign will not appear. The second possibility is to substitute $c_{\varepsilon,h}(s)$ in Eq.~\eqref{eq_remark_2} with
		$c_{\varepsilon,h}(t)$ in Eq.~\eqref{eq_remark_1} to expand the integrand, obtaining
		\begin{align} \nonumber
			c_{\varepsilon,h}(t^{n+1}) = & c_{\varepsilon,h}(t^n) +  \int_{t^n}^{t^{n+1}}\left(L_h + Q_h \vec{u}_h(s/\varepsilon)\right) \\  &\left( c_{\varepsilon,h}(t^{n})+\int^{s}_{t^n}\left(L_h + Q_h \vec{u}_h(\sigma/\varepsilon)\right) c_{\varepsilon,h}(\sigma)d\sigma\right)ds,
		\end{align} 
		\begin{comment}
			and the second order approximation will read
			\begin{eqnarray} \nonumber
				c_{\varepsilon,h}^{n+1} = c_{\varepsilon,h}^n + c_{\varepsilon,h}^n \int_{t^n}^{t^{n+1}}\left(L_h + Q_h \vec{u}_h(s/\varepsilon)\right)\left( 1 +\int^{s}_{t^n}\left(L_h + Q_h \vec{u}_h(\sigma/\varepsilon)\right) d\sigma\right)ds,
			\end{eqnarray}     
		\end{comment}
		in which, again, we do not obtain a negative sign. 
		%\giovanni{non ho capito: in che senso abbiamo "negative sign": nell'integrale $s$ va da $t^n$ a $t^{n+1}$, quindi sarebbe più logico scrivere }
		%\textcolor{blue}{
			%\[
			% c_{\varepsilon,h}^{n+1} = c_{\varepsilon,h}^n + c_{\varepsilon,h}^n \int_{t^n}^{t^{n+1}}\left(L_h + Q_h \vec{u}_h(s/\varepsilon)\right)\left( 1 \textcolor{red}{+\int^{s}_{t^n}}\left(L_h + Q_h \vec{u}_h(\sigma/\varepsilon)\right) d\sigma\right)ds,
			%\]
			%}
		%\giovanni{e non c'è alcun negative sign. Io toglierei questa parte, oppure direi che ancora non va bene, ma la questione del negative sign non la capisco proprio}
		
		The third possibility is to report Eq.~\eqref{eq_remark_2} in Eq.~\eqref{eq_remark_2}, and it is the one that we show in Proposition~\ref{pro_2nd}.
	\end{remark}
	
	%\giovanni{non capisco questo remark. Da dove vengono queste relazioni?}
	\begin{comment}
		\begin{remark}
			The following quantities in Eq.~\eqref{eq_2nd_order_method}:
			\[
			\mathcal{Q}_1 := Q_h\int _{t^n}^{t^{n+1}}\vec{u}_h(s)\,ds, \qquad \mathcal{Q}_2 :=  Q_h^2 \int_{t^n}^{t^{n+1}}\vec{u}_h(s)\int_{s}^{t^{n+1}}\vec{u}_h(\sigma)\,d\sigma\,ds
			\]
			are related by: $\, \mathcal{Q}_2 = \frac{1}{2} \, \mathcal{Q}_1^2 \equiv \frac{1}{k!}\mathcal{Q}_1^k \,$ that reminds the formula of the coefficients of the Taylor expansion. It is also valid for the quantities $L_h\Delta t$ and $ \frac{1}{2}{L_h^2\Delta t^2}$.
			%Indeed, if we define the function $\mathcal{G}(s) = \int_{s}^{t^{n+1}}\vec{u}_h(\sigma)\,d\sigma$, the following holds true
			%\begin{eqnarray*}
			%	Q_h^2 \int_{t}^{t^{n+1}}\vec{u}_h(s)\int_{s}^{t^{n+1}}\vec{u}_h(\sigma)\,d\sigma\,ds &=&
			%	Q_h \int_{t}^{t^{n+1}}\vec{u}_h(s)\mathcal{G}(s)\,ds = 
			%\int_{t}^{t^{n+1}}\mathcal{G}'(s)\mathcal{G}(s)\,ds = \\ \int_{t}^{t^{n+1}}\frac{d}{ds}\frac{1}{2}\mathcal{G}^2(s)\,ds &=& 
			%	\frac{1}{2}\mathcal{G}^2(t) =
			%	 \frac{1}{2}\left(\int_{s}^{t^{n+1}}\vec{u}_h(\sigma)\,d\sigma\right)^2.
			%\end{eqnarray*}    
		\end{remark}
	\end{comment}
	\section{Numerical results}
	In this section we confirm the numerical orders of convergence, uniformly in $\varepsilon$, of the integral-type schemes in Eqs.~(\ref{eq_scheme_1st_1}-\ref{eq_scheme_1st_2}) and (\ref{eq_scheme_2nd_1}-\ref{eq_scheme_2nd_2}). We also show different tests, changing the expression for the velocity $\vec{u}_{\varepsilon,h}$, and we confirm our previous assumptions about the space oscillations, in Remark~\ref{remark_space_osc}. 
	
	The domain considered in the following tests is $\Omega = \left( [-1,1]\times[-1,1]\right)\setminus \mathcal{B}$, where $\mathcal{B}$ is a circle centered in $(0,0)$ with radius $R_\mathcal{B} = 0.2$. 
	The expression for the initial condition used in the computations is the following
	\begin{equation}\label{equation_IC}
		c^0_{\varepsilon,h} = c_{\varepsilon,h}(t=0,x_h,y_h)=%\frac{1}{\left(2\pi\sigma^2\right)^{3/2}}
		\exp\left(-\left(x_h^2+(y_h-y_0)^2\right)/2\sigma^2\right), \quad y_0,\sigma \in \mathbb R.
	\end{equation}
	
	In our tests we choose the following expressions for the velocity $\vec{u}_h(t/\varepsilon)$
	\begin{align}
		\textsc{TestCos}: \qquad\vec{u}_h & = A R_\mathcal{B} \cos(2\pi\,t/\varepsilon) 
		\cdot (x_h,y_h)^\top
		%\begin{pmatrix}
		%x_h \\
		%y_h
		%\end{pmatrix}
		/\left(x_h^2+y_h^2\right) %\\
		%\textsc{TestCos-Sin}: \qquad\vec{u}_h &= A (a_1\cos(2\pi\,t/\varepsilon) + a_2\sin(2\pi\,t/\varepsilon)) 
		%\cdot (x_h,y_h)^\top
		%\begin{pmatrix}
		%x_h \\
		%y_h
		%\end{pmatrix}
		%/\sqrt{x_h^2+y_h^2}
	\end{align}
	where $A$ is the amplitude. To show the oscillations in space, we define a test function that depends on ${x}_h/\varepsilon$ 
	\begin{equation}
		\label{eq_veloc_osc_space}
		\textsc{TestOsc}: \qquad    \vec{u}_h = A R_\mathcal{B}\cos(2\pi\,(t+x_h)/\varepsilon) \cdot (1,0)^\top.
	\end{equation}
	
	To calculate the error, we first compute a reference solution $c_{\varepsilon,h}^{\rm ref}$, choosing related reference time step $\Delta t_{\rm ref}$, number of points $N_{\rm ref}$ and final time $t_{\rm fin}$, for the following set of $\varepsilon = 10^{-k}, k \in \{0,1,2,3,4,5,6\}$. Then, we calculate different solutions $c^{\Delta t,N}_{\varepsilon,h}$, for different $\Delta t = 0.01\cdot 2^{-k}, k\in \{0,1,2,3,4,5,6\}$ and $N = 20,40,80,160$. After computing all the solutions, we calculate the $L^2-$norm of the relative error as follow:
	\begin{equation}
		{\rm error} = \frac{||c_{\varepsilon,h}^{\rm ref} - c^{\Delta t,N}_{\varepsilon,h}||_2}{||c_{\varepsilon,h}^{\rm ref}||_2}.
	\end{equation}
	\begin{figure}[htp]
		\centering
		\hfill
		\begin{minipage}
			{.45\textwidth}
			\centering		%\includegraphics[width=\textwidth]{Figures/accuracy_time_indip_eps.png}
			\begin{overpic}[abs,width=\textwidth,unit=1mm,scale=.25]{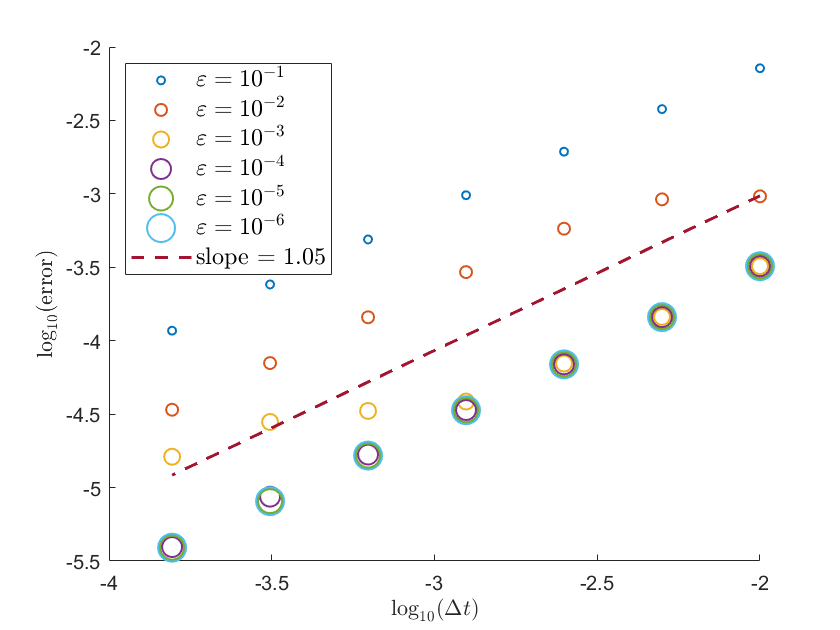}
				'\put(-1.5,50){(a)}
			\end{overpic}	
		\end{minipage}
		\begin{minipage}
			{.45\textwidth}
			\centering
			\begin{overpic}[abs,width=\textwidth,unit=1mm,scale=.25]{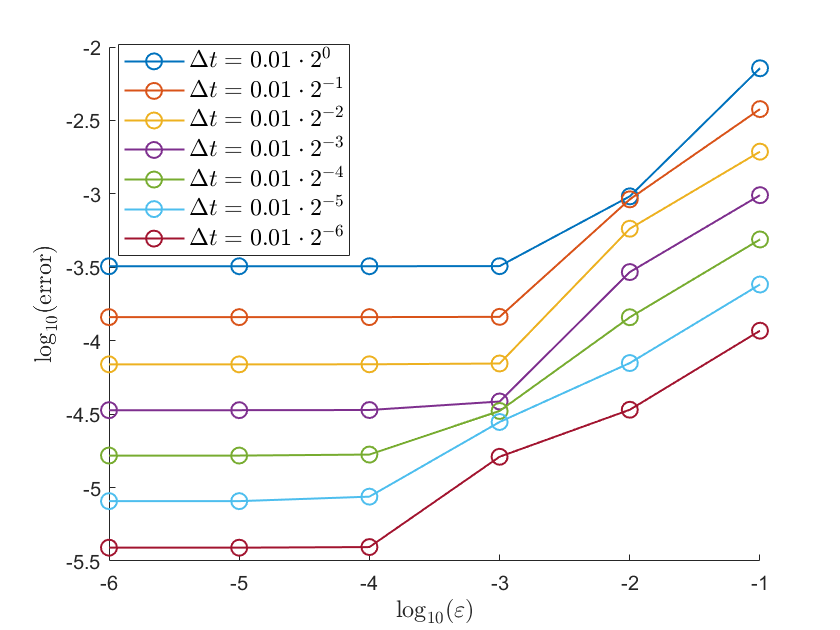}
				'\put(-1.5,50){(b)}
			\end{overpic}
		\end{minipage} \hspace*{\fill}
		\caption{\textit{Plot of the relative $L^2$ error for the first order scheme in Eq.~\eqref{eq_1st_ord_scheme} as a function of $\Delta t$ and different $\varepsilon$ (a) and as a function of $\varepsilon$ and different $\Delta t$ (b). The parameters of the tests are: $t_{\rm fin}=0.1, \Delta t_{\rm ref} = 10^{-5},  N_{\rm ref} = 160, D = 0.02, x_0 = 1.2, y_0 = 0, \sigma = 0.2, \delta = 10^{-2}, A = 1$. The initial condition is defined in Eq.~\eqref{equation_IC} and the velocity in \textsc{TestCos}}.}
		\label{fig_time_1st_order}
	\end{figure}
	
	In Fig.~\ref{fig_time_1st_order} (a) we show the $L^2-$norm of the error, as a function of $\Delta t$ and for different values of $\varepsilon$. The method that we consider here is the first order accurate numerical scheme defined in Eq.~\eqref{eq_1st_ord_scheme}. In Fig.~\ref{fig_time_1st_order} (b) we have the $L^2-$norm of the error as a function of $\varepsilon$, for different values of $\Delta t$, to show that the numerical scheme is uniformly accurate in $\varepsilon$. Analogously, in Fig.~\ref{fig_time_2nd_order} (a) and (b) we show that the numerical scheme defined in Eq.~\eqref{eq_2nd_order_method} is second order accurate, uniformly in $\varepsilon$. 
	\begin{figure}[htp]
		\centering
		\hfill
		\begin{minipage}
			{.45\textwidth}
			\centering		%\includegraphics[width=\textwidth]{Figures/accuracy_time_indip_eps.png}
			\begin{overpic}[abs,width=\textwidth,unit=1mm,scale=.25]{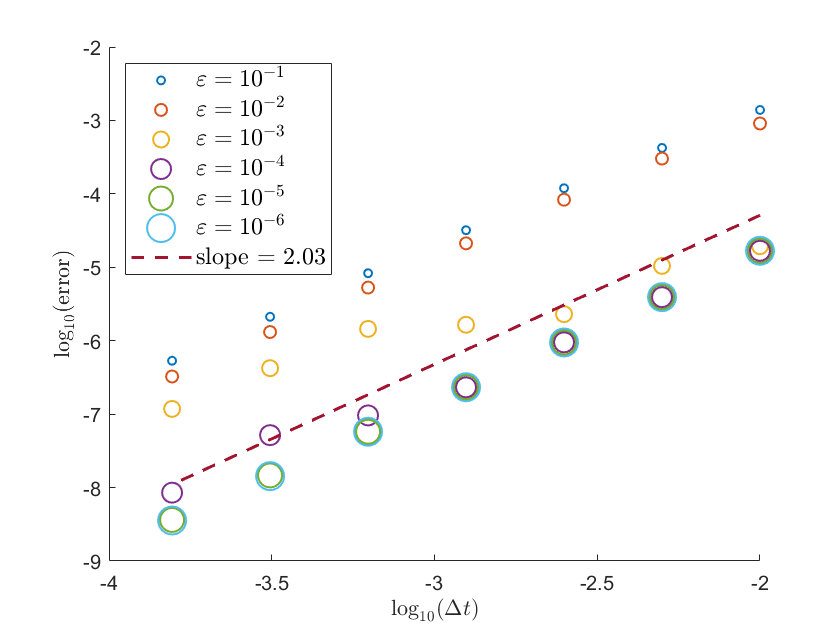}
				'\put(-1.5,50){(a)}
			\end{overpic}	
		\end{minipage}
		\begin{minipage}
			{.45\textwidth}
			\centering
			\begin{overpic}[abs,width=\textwidth,unit=1mm,scale=.25]{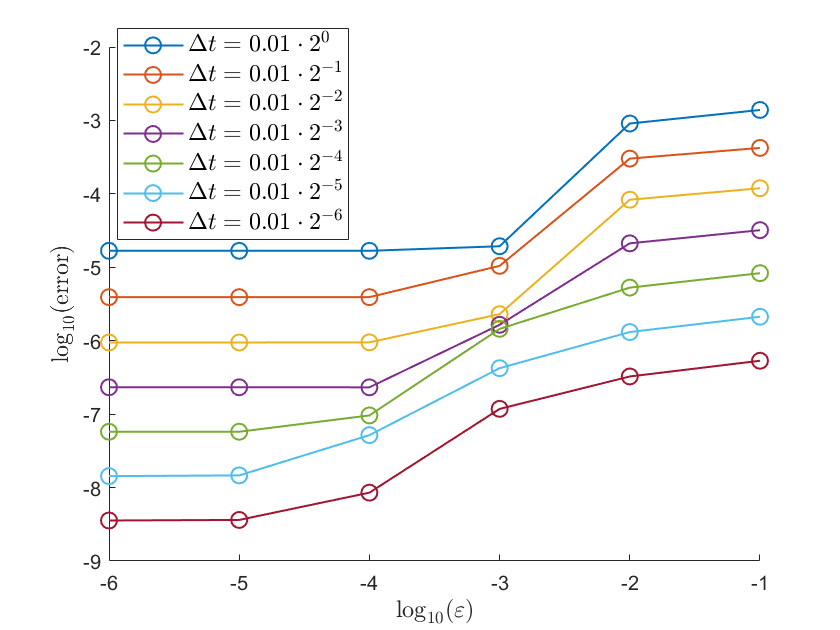}
				'\put(-1.5,50){(b)}
			\end{overpic}
		\end{minipage}
		\hspace*{\fill}
		\caption{\textit{Plot of the relative $L^2$ error for the second order scheme in Eq.~\eqref{eq_2nd_order_method} as a function of $\Delta t$ and different $\varepsilon$ (a) and as a function of $\varepsilon$ and different $\Delta t$ (b). The parameters of the tests are: $t_{\rm fin}=0.1, \Delta t_{\rm ref} = 10^{-5},  N_{\rm ref} = 160, D = 0.02, y_0 = -1, \sigma = 0.2, \delta = 10^{-2}, A = 1$. The initial condition is defined in Eq.~\eqref{equation_IC} and the velocity in \textsc{TestCos}}.}
		\label{fig_time_2nd_order}
	\end{figure}
	
	In Fig.~\ref{fig_space_2nd_order} we show the $L^2-$norm of the error of the space discretization defined in Eq.~\eqref{pde2d}. In (a) we show the error as a function of $\Delta x$ for different $\varepsilon$, and in (b) there is the error as a function of $\varepsilon$, for different $\Delta x$. As expected, the numerical scheme is second order accurate, uniformly in $\varepsilon$. 
	
	Fig.~\ref{fig_detector_cos} represents the time evolution of the solution at the detector using the  second order scheme in Eq.~\eqref{eq_2nd_order_method}, for the \textsc{TestCos}. We choose two different values of $\varepsilon$: in panel (a) $\varepsilon = 0.01$ and in panel (b) $\varepsilon = 0.001$. We show a reference solution, $c_{\varepsilon,h}^{\rm ref}$ (blue line) with $\Delta t_{\rm ref} = 10^{-4}$, together with different solutions for different time steps. The main goal is to show that for $\Delta t > \varepsilon$, the two concentrations, $c_{\varepsilon,h}^{\rm ref}$ and $c_{\varepsilon,h}^{\Delta t,N}$, overlap.
	
	In Fig.~\ref{fig_CN_01}--\ref{fig_CN_001} (a), we show the comparison between the second order numerical scheme in Eq.~\eqref{eq_2nd_order_method} (yellow stars and red diamonds), and a direct simulation using a traditional Crank-Nicolson (CN) method (green circles), together with a reference solution (blue line). The reference solutions have been computed with a time step $\Delta t_{\rm ref} = 2\times 10^{-5}$, which is much smaller than the period $\varepsilon$. The other solutions are obtained using a time step of length which is not commensurable with the period, but is of the same order of magnitude of $\varepsilon$. In the zoom--in panels (b) and (c) we see the advantages of the scheme described in this paper (green stars, red diamonds).
	The advantage of the numerical scheme described in this paper is more evident for small values of $\varepsilon$, as we show in Fig.~\ref{fig_CN_001}, with $\varepsilon = 10^{-2}$: for $\varepsilon = 10^{-1}$ accurate solutions are obtained with a time step greater than $3\varepsilon$, while in the case $\varepsilon = 10^{-2}$, the method provides an accurate solution even when the time step is one order of magnitude larger than the period  (red diamonds).

	As mentioned in the Remark~\ref{remark_space_osc}, when there is a dependence on $\vec x/\varepsilon$, spatial oscillations occur, rendering the Propositions presented in this paper not valid. In Fig.~\ref{fig_space_osc} we show the results of the \textsc{TestOsc}, considering time steps larger than $\varepsilon$ (green circles) together with a reference solution (i.e., $N_{\rm ref} = 200, \Delta t_{\rm ref} = 10^{-3}$) to confirm that the accuracy is not guaranteed when we use the time integrator in \eqref{eq_2nd_order_method}.
	\begin{figure}[htp]
		\centering
		\hfill
		\begin{minipage}
			{.45\textwidth}
			\centering		%\includegraphics[width=\textwidth]{Figures/accuracy_x_indip_eps2.png}
			\begin{overpic}[abs,width=1\textwidth,unit=1mm,scale=.25]{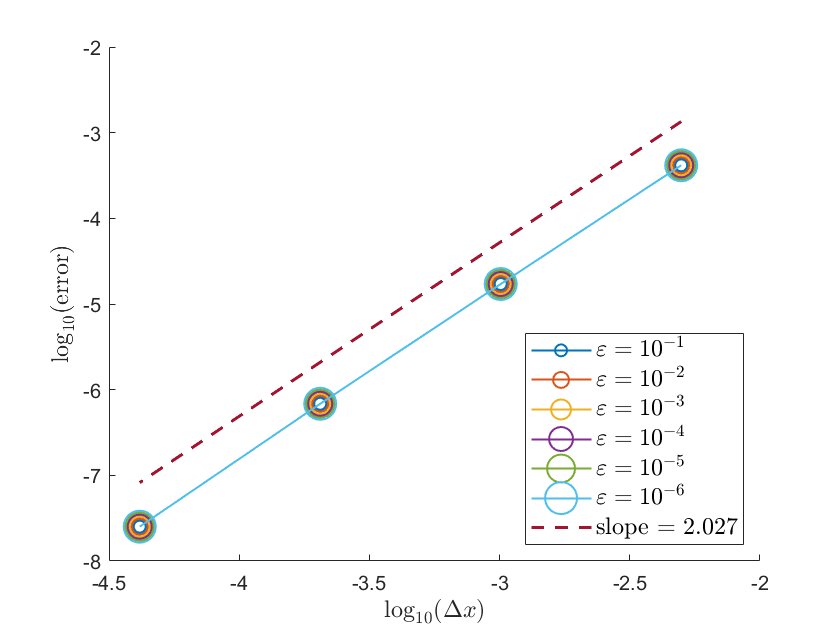}
				'\put(-1.5,50){(a)}
			\end{overpic}	
		\end{minipage}\hfill
		\begin{minipage}
			{.45\textwidth}
			\centering
			\begin{overpic}[abs,width=\textwidth,unit=1mm,scale=.25]{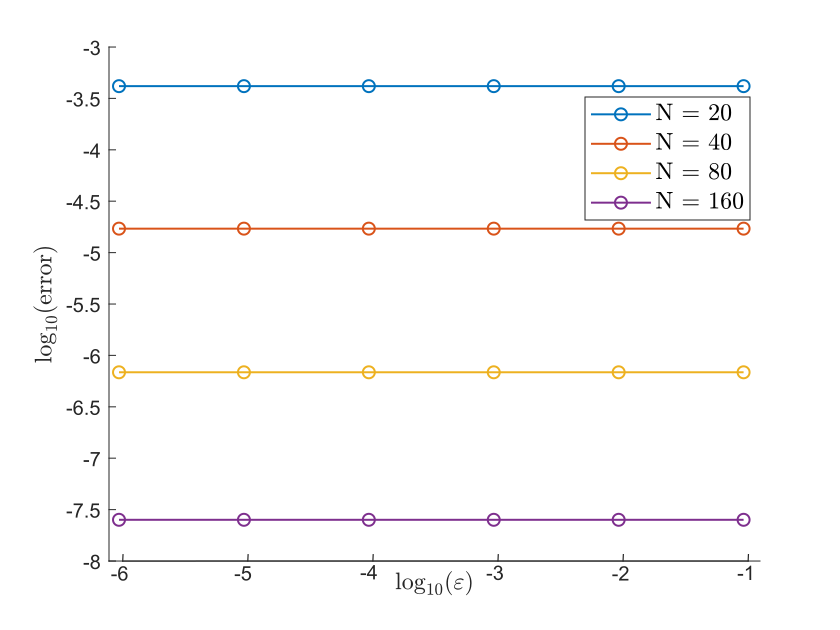}
				'\put(-1.5,50){(b)}
			\end{overpic}	 
		\end{minipage}
		\hspace*{\fill}
		\caption{\textit{Plot of the relative $L^2$ error in space for the second order scheme in Eq.~\eqref{pde2d} as a function of $\Delta x$ and different $\varepsilon$ (a) and as a function of $\varepsilon$ and different $\Delta x$ (b). The parameters of the tests are: $t_{\rm fin}=0.1, \Delta t_{\rm ref} = 10\cdot 2^6, N_{\rm ref} = 640, D = 0.02, y_0 = -1, \sigma = 0.2, \delta = 10^{-2}, A = 1$. The initial condition is defined in Eq.~\eqref{equation_IC} and the velocity in \textsc{TestCos}}.}	\label{fig_space_2nd_order}
	\end{figure}
	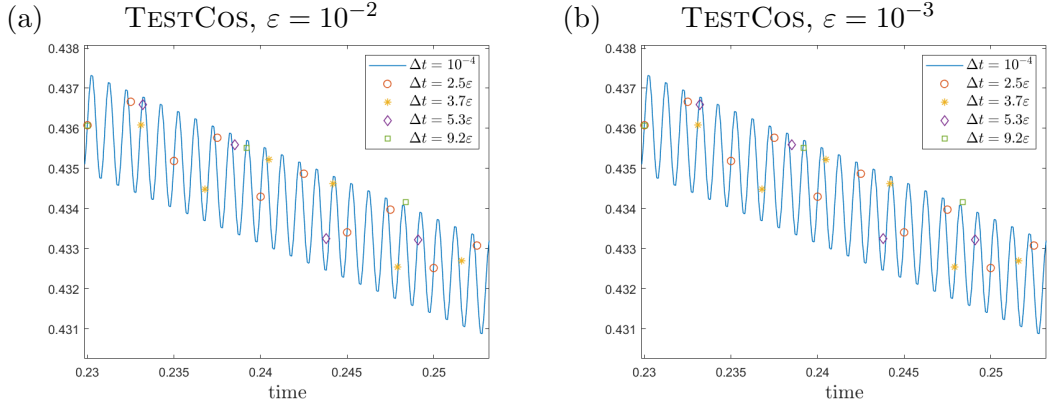
\begin{figure}[htp]
		\centering
		\hfill
		\begin{minipage}
			{.45\textwidth}	\centering
			\begin{overpic}[abs,width=\textwidth,unit=1mm,scale=.25]{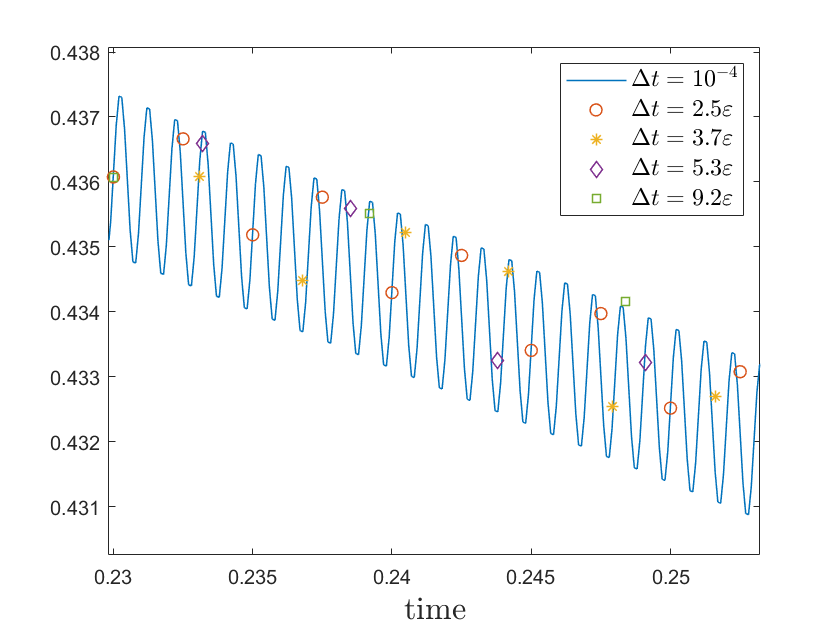}
				\put(-1.5,50){(a)}
				\put(14.,50){\textsc{TestCos}, $\varepsilon = 10^{-2}$}
			\end{overpic}
		\end{minipage}\hfill
		\begin{minipage}
			{.45\textwidth}	\centering
			\begin{overpic}[abs,width=\textwidth,unit=1mm,scale=.25]{Figures/detector_N160_Cos_eps001.png}
				\put(-1.5,50){(b)}
				\put(14.,50){\textsc{TestCos}, $\varepsilon = 10^{-3}$}
			\end{overpic}
		\end{minipage}
		\hspace*{\fill}
		\caption{\textit{Concentration $c_{\varepsilon,h}$ at the detector  for different values of $\Delta t$ with $\varepsilon=0.01$ (a) and $\varepsilon = 0.001$ in (b), for \textsc{TestCos}. The initial condition is defined in Eq.~\eqref{equation_IC} and the detector is centered in $P=(0,-0.5)$, $P\in\Omega$. The parameters of the test are: $N = 160, A = 10, \delta = 10^{-3}, y_0 = -0.75$ and $D = 0.02$.}}
		\label{fig_detector_cos}
	\end{figure}
	
	% \begin{figure}[htp]
		% 	\centering
		% 	\hfill
		% 	\begin{minipage}
			% 		{.45\textwidth}	\centering
			% \begin{overpic}[abs,width=\textwidth,unit=1mm,scale=.25]{Figures/detector_N160_CosSin_eps001.png}
				% '\put(-1.5,50){(a)}
				% \put(10.,43){\textsc{TestCos-Sin}, $\varepsilon = 10^{-3}$}
				% \end{overpic}
			% 	\end{minipage}\hfill
		% 	\begin{minipage}
			% 		{.45\textwidth}	\centering
			% %\includegraphics[width=\textwidth]{Figures/detector_eps1_Sin_zoom.png}
			% \begin{overpic}[abs,width=\textwidth,unit=1mm,scale=.25]{Figures/detector_N160_CosSin_eps0001.png}
				% '\put(-1.5,50){(b)}
				% \put(10.,43){\textsc{TestCos-Sin}, $\varepsilon = 10^{-4}$}
				% \end{overpic}
			% 	\end{minipage}
		% 	\hspace*{\fill}
		% 	\caption{\textit{Detector values of the concentration for different values of $\Delta t$ with $\varepsilon=0.01$ (a) and $\varepsilon = 0.001$ in (b). The initial condition is defined in Eq.~\eqref{equation_IC} and the detector is centered at $P=(0,-0.5)$, $P\in\Omega$. The parameters of the test are: $N = 160, A = 20, a_1 = 0.6, a_2 = 0.4, \delta = 10^{-3}, y_0 = -0.75$, and $D = 0.01$.}}
		% 	\label{fig_detector_cos_sin}
		% \end{figure}
	
	\begin{figure}[htp]
		\centering
		\hfill
		\begin{minipage}[b]
			{.32\textwidth}
			\centering
			\begin{overpic}[abs,width=\textwidth,unit=1mm,scale=.25]{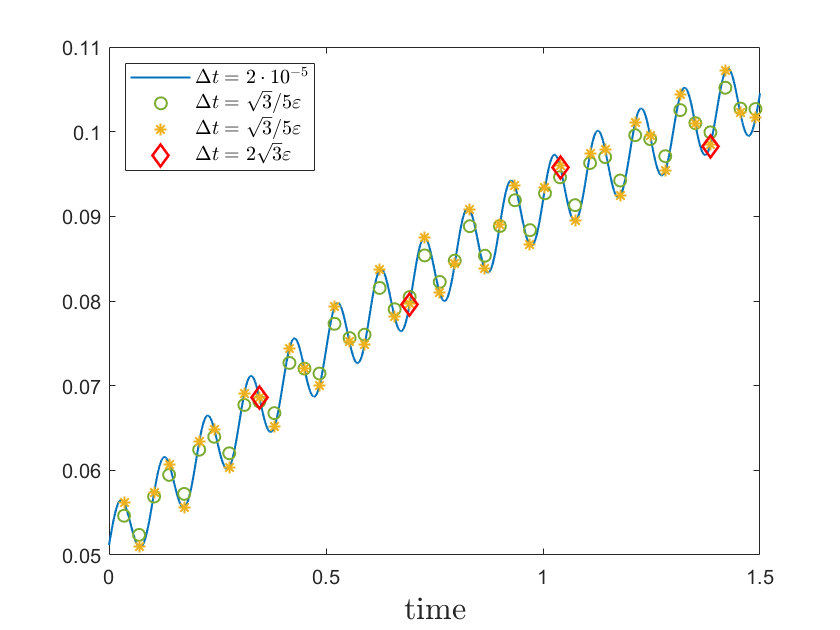}
				\put(-1.5,35){(a) \textsc{TestCos}, $\varepsilon = 10^{-1}$}  
			\end{overpic}
		\end{minipage}\hfill
		\begin{minipage}[b]
			{.32\textwidth}
			\centering
			\begin{overpic}[abs,width=\textwidth,unit=1mm,scale=.25]{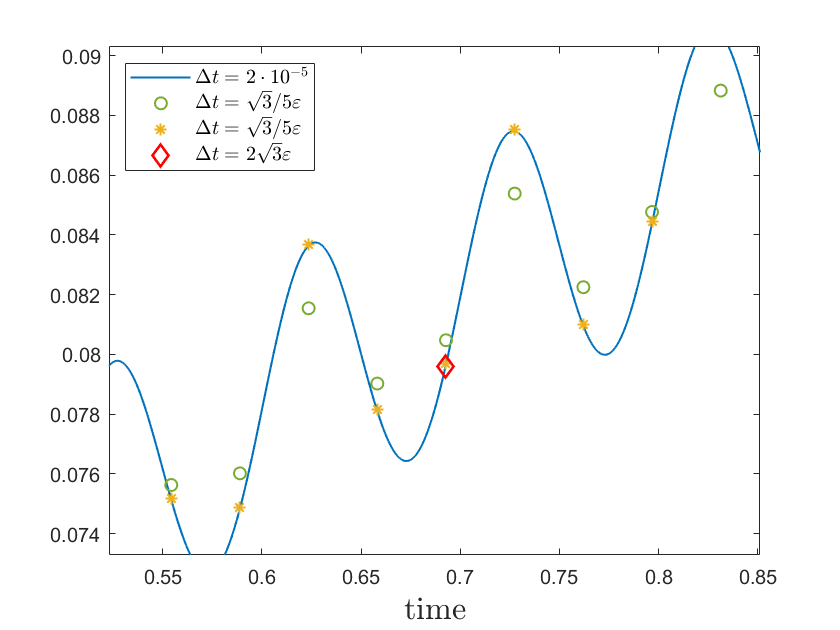}
				\put(-1.5,35){(b) \textsc{TestCos}, $\varepsilon = 10^{-1}$}  
			\end{overpic}
		\end{minipage}\hfill
		\begin{minipage}[b]
			{.32\textwidth}
			\centering
			\begin{overpic}[abs,width=\textwidth,unit=1mm,scale=.25]{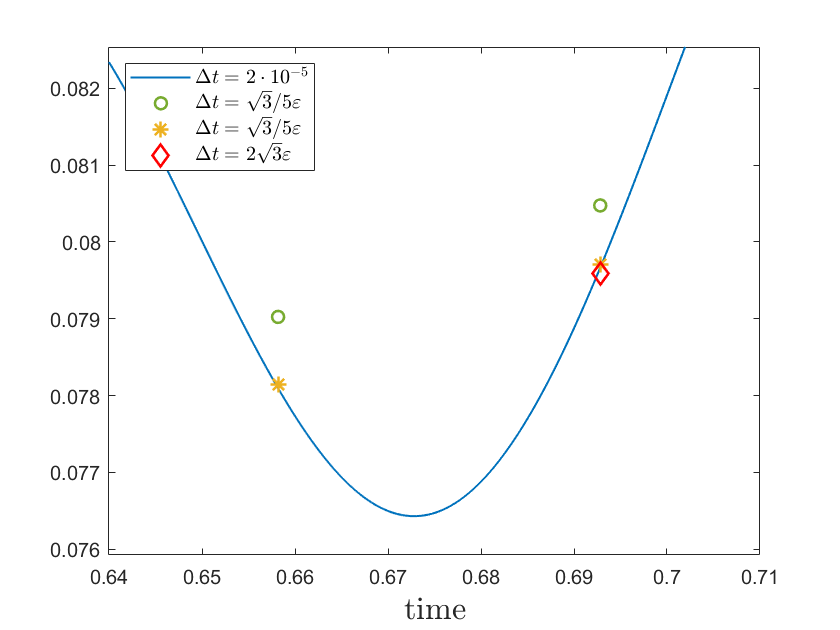}
				\put(-1.5,35){(c) \textsc{TestCos}, $\varepsilon = 10^{-1}$}  
			\end{overpic}
		\end{minipage}\hfill
		\caption{\textit{(a): Comparison of two different, second order accurate in time integrators for Eq.~\eqref{pde2d} for \textsc{TestCos}, together with a reference solution (blue line): integral-type scheme in Eq.~\eqref{eq_2nd_order_method}  (yellow stars and red diamonds) and Crank-Nicolson scheme (CN) (green circles). (b): Zoom--in of panel (a). (c): Zoom--in of panel (b). The parameters of the system are: $\varepsilon = 10^{-1}, N = 160, A = 1, D = 0.02, \delta = 10^{-1}, y_0 = -0.75$. }}
		\label{fig_CN_01}
	\end{figure}

	\begin{figure}[htp]
		\centering
		\hfill
		\begin{minipage}[b]
			{.32\textwidth}
			\centering
			\begin{overpic}[abs,width=\textwidth,unit=1mm,scale=.25]{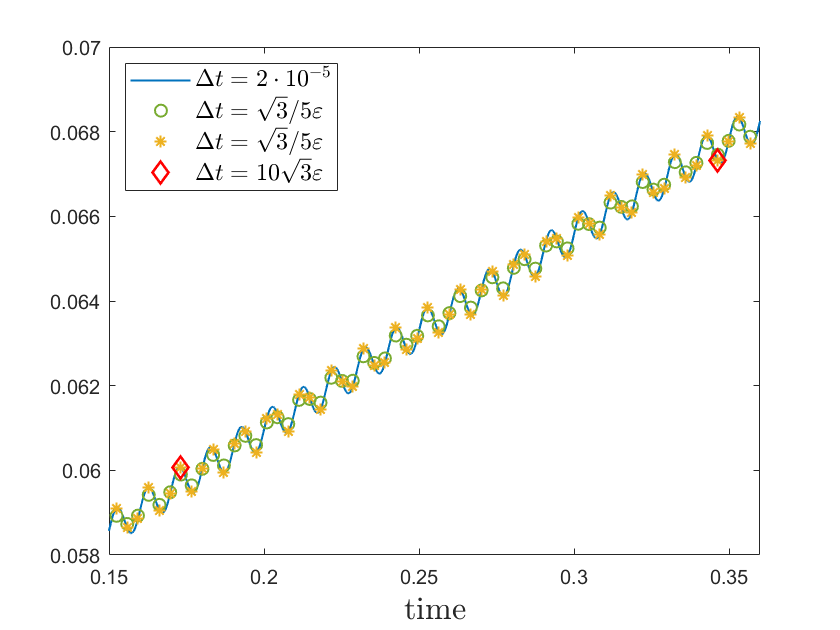}
				\put(-1.5,35){(a) \textsc{TestCos}, $\varepsilon = 10^{-2}$}  
			\end{overpic}
		\end{minipage}\hfill
		\begin{minipage}[b]
			{.32\textwidth}
			\centering
			\begin{overpic}[abs,width=\textwidth,unit=1mm,scale=.25]{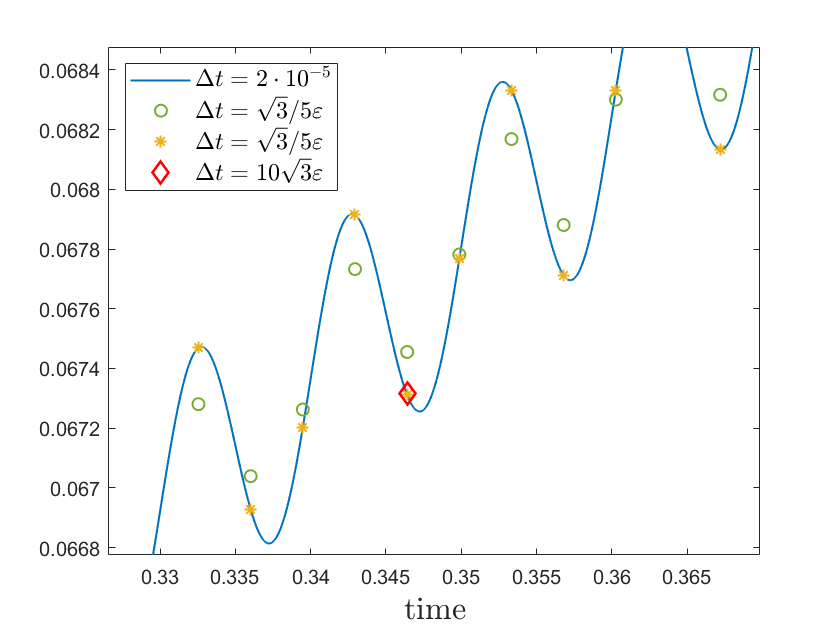}
				\put(-1.5,35){(b) \textsc{TestCos}, $\varepsilon = 10^{-2}$}  
			\end{overpic}
		\end{minipage}\hfill
		\begin{minipage}[b]
			{.32\textwidth}
			\centering
			\begin{overpic}[abs,width=\textwidth,unit=1mm,scale=.25]{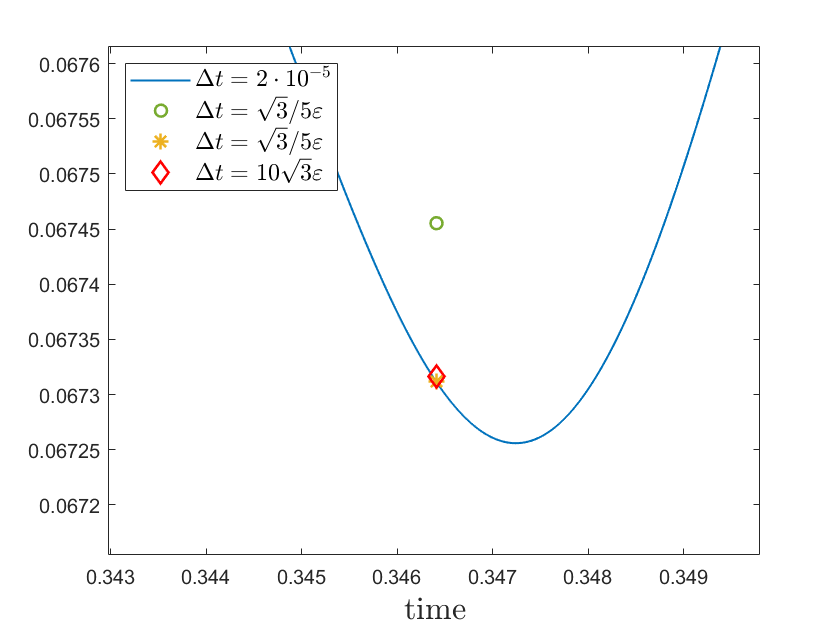}
				\put(-1.5,35){(c) \textsc{TestCos}, $\varepsilon = 10^{-2}$}  
			\end{overpic}
		\end{minipage}\hfill
		\caption{\textit{(a): Comparison of two different, second order accurate in time integrators for Eq.~\eqref{pde2d} for \textsc{TestCos}, together with a reference solution (blue line): integral-type scheme in Eq.~\eqref{eq_2nd_order_method}  (yellow stars and red diamonds) and Crank-Nicolson scheme (CN) (green circles). (b): Zoom--in of panel (a). (c): Zoom--in of panel (b). The parameters of the system are: $\varepsilon = 10^{-2}, N = 160, A = 1, D = 0.02, \delta = 10^{-2}, y_0 = -0.75$. }}
		\label{fig_CN_001}
	\end{figure}
	
	\begin{figure}
		\centering
		\includegraphics[width=0.5\textwidth]{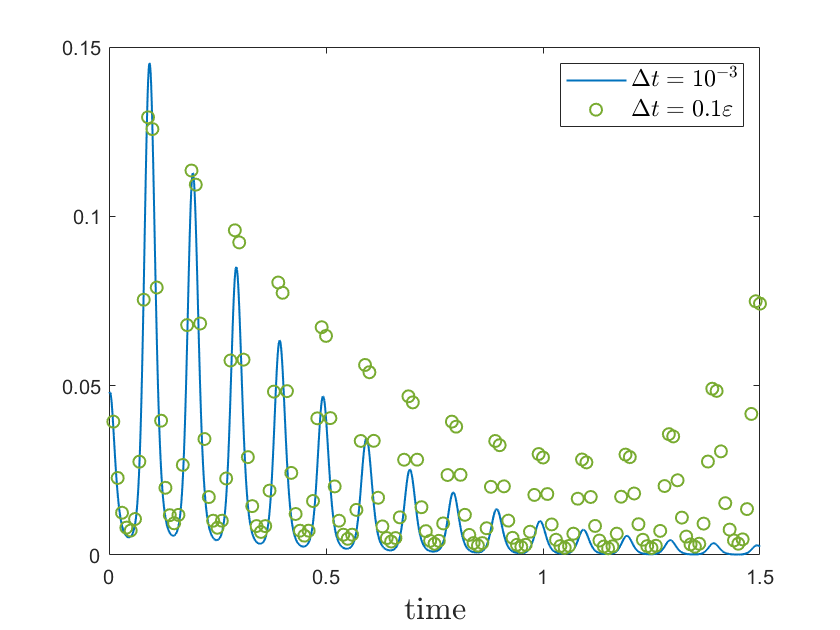}
		\caption{\textit{Detector values of the concentration $c_{\varepsilon,h}$ for \textsc{TestOsc}, and different time steps $\Delta t$. The initial condition is defined in Eq.~\eqref{equation_IC} and the detector is centered at $P=(0,-0.5)$, $P\in\Omega$. Parameters of the numerical test: $N = 200, \varepsilon = 0.1, A = 1, y_0 = -1, \Delta t = 10^{-3}, D = 0.1$ and $\delta = 10^{-3}$.}}
		\label{fig_space_osc}
	\end{figure}
	
	\section{Conclusions}
	In this work, we have presented a general strategy for constructing uniformly accurate numerical schemes for highly oscillatory advection diffusion equations, motivated by the necessity of solving the multiscale problem developed in \cite{astuto2023multiscale} to treat the diffusion and trapping or a surfactant around a rapidly oscillatory bubble (see also  \cite{CiCP-31-707}, \cite{COCO2020109623}, \cite{ASTUTO2023111880}). 
	
	Traditional numerical time integrators, that are commonly employed, require significant time restrictions to achieve accurate solutions when the oscillation period is much smaller than the diffusion time, since the time step has to be smaller than the oscillation period. 
	
	Here we demonstrate that there are alternative choices available to achieve the desired accuracy without resolving the short oscillation period, thus making the computational cost orders of magnitude smaller that the one required by methods from the literature. 
	
	The accuracy of the proposed method actually improves when the oscillation frequency increases, providing accurate solutions with time steps much larger then the oscillation period. 
	
	Here we assume that the fluid velocity is a known function of space and time. 
	A natural extension of the current work 
	is the numerical solution of a coupled Stokes-advection-diffusion system, where the expression for the velocity in the advection term is obtained numerically by solving the Stokes equations.

\end{document}